\pdfoutput=1

\documentclass[a4paper,UKenglish,cleveref,autoref,thm-restate,envcountsame,envcountsect,numberwithinsect,runningheads]{llncs}
\makeatletter
\renewcommand*\l@author[2]{}
\renewcommand*\l@title[2]{}
\makeatletter

\usepackage[svgnames,rgb]{xcolor}
\usepackage{graphicx}
\usepackage{enumitem}

\usepackage{amsthm}

\usepackage{proof} %
\usepackage{mathtools} %
\usepackage{amssymb} %
\usepackage{pgfplots}
\pgfplotsset{compat=1.13}
\usepackage{tikz} %
\usepackage{tikz-cd} %
\usepackage{circuitDiagrams}
\usepackage{feedbackdiagrams}
\usepackage{pointedfeedbackdiagrams}
\usepackage{macros}
\usepackage{stringdiagrams}
\usepackage{wrapfig}
\usepackage{subcaption}
\usepackage{nccmath}
\usepackage{hyperref}
\usepackage{cleveref}
\usepackage{adjustbox}

\allowdisplaybreaks

\PassOptionsToPackage{breaklinks}{hyperref}
\usepackage{xurl}
\usepackage{hyperref} %
\usepackage[hyphenbreaks]{breakurl}

\usepackage[all]{hypcap}
\usepackage{float} %
\hypersetup{
  colorlinks = true,
  linkcolor = NavyBlue,
  citecolor = Purple,
  urlcolor = NavyBlue,
}
\makeatletter
\newcommand{\labeltext}[3][]{%
    \@bsphack%
    \csname phantomsection\endcsname%
    \def\tst{#1}%
    \def\refmarkup{}%
    \ifx\tst\empty\def\@currentlabel{\refmarkup{#2}}{\label{#3}}%
    \else\def\@currentlabel{\refmarkup{#1}}{\label{#3}}\fi%
    \@esphack%
    #2%
}
\makeatother

\newtheorem*{definition*}{Definition}
\newtheorem*{theorem*}{Theorem}

\title{Span(Graph): a Canonical Feedback Algebra of Open Transition Systems \thanks{Di Lavore, Román and Sobociński were supported by the European Union
 through the ESF funded Estonian IT Academy research measure (2014-2020.4.05.19-0001). This work was also supported by the Estonian Research
Council grant PRG1210.}}

\author{
  Elena Di Lavore \inst{1} \and
  Alessandro Gianola \inst{2} \and
  Mario Román \inst{1} \and
  Nicoletta Sabadini \inst{3}\and
  Paweł Sobociński \inst{1}}
\authorrunning{Di Lavore, Gianola, Román, Sabadini, Sobociński}

\institute{
  Tallinn University of Technology, Ehitajate tee 5, 12616 Tallinn, Estonia
  \and
  Free University of Bozen-Bolzano, Piazza Domenicani, 3, 39100 Bolzano BZ, Italy
  \and
  Universit\`a degli Studi dell'Insubria, Via Ravasi, 2, 21100 Varese VA, Italy
}

\begin{document}

\maketitle

\begin{abstract}
  We show that \(\SpanGraphV\), an algebra for \emph{open transition systems} introduced by Katis, Sabadini and Walters, satisfies a universal property. By itself, this is a justification of the canonicity of this model of concurrency.
  However,
  the universal property is itself of interest, being a formal demonstration of the relationship between feedback and state.

  Indeed, \emph{feedback} categories, also originally proposed by Katis, Sabadini and Walters, are a weakening of traced monoidal categories, with various applications in computer science. A \emph{state bootstrapping} technique, which has appeared in several different contexts, yields \emph{free} such categories.

  We show that \(\SpanGraphV\) arises in this way, being the
  free feedback category over \(\Span(\Set)\). Given that the latter can be seen as an algebra of predicates, the algebra of open transition systems thus arises -- roughly speaking -- as the result of bootstrapping state to that algebra.

  Finally, we generalize feedback categories endowing state spaces with extra structure: this extends the framework from mere transition systems to automata with initial and final states. 
\end{abstract}
\setcounter{tocdepth}{2}
\tableofcontents
\newpage

\section{Introduction}\label{sec:intro}
Software engineers need models. In fact, models developed
in the early years of computer science have been extremely influential on the emergence of software engineering as a discipline. Prominent examples include flowcharts and state machines, and a part of the reason for their impact and longevity is the fact that they are underpinned by relevant and well-understood mathematical theories.

However, while \emph{concurrent} software has been intensively studied since the early 60s, the theoretical research landscape remains quite fragmented. Indeed,
Abramsky~\cite{abramsky14} argues that the reason for the proliferation of models, their sometimes overly locally-optimised techniques, and the difficulty of understanding and relating their expressivity, is the fact that we still do not have a satisfactory understanding of the underlying mathematical principles of concurrency.

A way to identify such principles and arrive at more canonical models is to look for logical or mathematical justifications. An example is the recent discovery and work on of Curry-Howard style connections between calculi for concurrency and fragments of linear logic, which guided the development of session types~\cite{dezani09}. Another possible route is to search for models that satisfy some \emph{universal property}.

The latter approach is the remit of this paper:
we focus on the $\SpanGraphV$ model of concurrency, introduced by Katis, Sabadini and Walters~\cite{katis97} as an algebra of \emph{open transition systems}, and
show that it satisfies a universal property:
it is the free \categoryWithFeedback{} over the category of spans of functions.

The free construction is in itself interesting and can be described as a kind of ``state-bootstrapping''. We thus position our main result within the theoretical context of \categoriesWithFeedback{}, their relationship with state, and the more restrictive---yet better known---notion of \tracedMonoidalCategories{}. Our exploration of this wider context is justified, given the panoply of related, yet partial, accounts in the literature.

\begin{wrapfigure}{r}{5cm}
  \centering
  \norlatchwithlabels
  \caption{NOR latch.}
  \label{fig:norintro}
\end{wrapfigure}

\medskip
The relationship between feedback and state is well-known by engineers. In fact,
a remarkable fact from electronic circuit design is how data-storing components
can be built out of a combination of \emph{stateless components} and
\emph{feedback}. A famous example is the (set-reset) ``NOR latch'': a circuit with two stable configurations that \emph{stores} one bit. %

The NOR latch is controlled by two inputs, $\mathsf{Set}$
and $\mathsf{Reset}$. Activating the first sets the output value to
$\mathsf{A} = \mathsf{1}$; activating the second makes the output value return
to $\mathsf{A} = \mathsf{0}$. This change is permanent: even when both
$\mathsf{Set}$ and $\mathsf{Reset}$ are deactivated, the feedback loop maintains
the last value the circuit was set to\footnote{In its original description:
  \emph{``the relay is designed to produce a large and \textbf{permanent} change
    in the current flowing in an electrical circuit by means of a small
    electrical stimulus received from the outside''} (\cite{eccles18}, emphasis
  added).}---to wit, a bit of data has been conjured out of thin air. %
The results of this paper allow one to see the latch as
an instance of a more
abstract phenomenon.

Indeed, there is a natural weakening of the notion of \tracedMonoidalCategories{} called
\emph{feedback categories} \cite{katis02}. The construction of the
\emph{free} feedback category coincides with a ``state-bootstrapping'' construction,
$\St(\bullet)$, that appears in several different contexts in the literature
\cite{bonchi19,hoshino14,sabadini95}. We recall this construction and its mathematical status (\Cref{th:storeisfree}), which can be summed up by the following intuition:
\[\mbox{Theory of Processes} + \mbox{Feedback} = \mbox{Theory of Stateful Processes}.\]

\medskip
The $\SpanGraph$ model of concurrency, introduced in~\cite{katis97}, is an algebra of \emph{communicating state machines}, or --- equivalently --- \emph{open transition systems}.

Let us first explain some terminology.
A span \(X \to Y\) in a category \(\catC\) is a pair of morphisms \(l \colon A \to X\) and \(r \colon A \to Y\) with a common domain (\Cref{def:span}).
When \(\catC\) has enough structure, spans form a category.
This is the case for the category of graphs \(\Graph\), where objects are graphs and morphisms are, intuitively, pairs of functions that respect the graph structure (\Cref{def:cat-graph}).
Summarizing the above, the morphisms of \(\SpanGraph\) are given by pairs of graph homomorphisms, \(l \colon G \to X\) and \(r \colon G \to Y\), with a common domain \(G\).
We think of a span of graphs as a transition system, the graph \(G\), with boundary interfaces \(X\) and \(Y\).

Open transition systems %
interact
by synchronization along a common boundary, producing a
 \emph{simultaneous change of state}.
This corresponds to a composition of spans, realized by taking a pullback in $\Graph$ (see \Cref{def:open-transition-system}).
The dual algebra of
$\Cospan(\Graph)$ was introduced in~\cite{katis00} (see \Cref{def:graph-discrete-boundaries}).

Informally, a morphism \(X \to Y\) of $\SpanGraph$ is a state machine with states and
transitions, i.e. a finite graph given by the `head' of the span. The transition system
is equipped with left and right interfaces or \emph{communication ports}, \(X\) and \(Y\), and every
transition is labeled by the effect it produces in \emph{all}
its interfaces. Let us focus on some concrete examples.

Let $\Bool=\{\,0,\,1\,\}$. We abuse notation by considering $\Bool$ as a single-vertex graph with two edges, corresponding to the \emph{signals} $0$ and $1$. Indeed, as we shall see in examples below, it is useful to think of single-vertex graphs as alphabets of signals available on interfaces.

In \Cref{fig:norandlatch}, we depict two open transition systems as arrows of $\SpanGraph$. The first represents a NOR gate $\Bool \times \Bool \to \Bool$.
To give an arrow of this type in $\SpanGraph$ is to give a span of graph homomorphisms \[\Bool \times \Bool \overset{l}{\longleftarrow} N \overset{r}{\longrightarrow} \Bool.\]
The graphical rendering (\Cref{fig:norandlatch}, left) is a compact representation of the components of this span: the \emph{unlabeled} graph in the bubble is $N$, and the labels witness the action of two homomorphisms, respectively $l \colon N\rightarrow \Bool\times\Bool$ and $r \colon N\rightarrow \Bool$.
Transitions represent the valid input/output configurations of the NOR gate. For example, the edge %
with label \((\binom{0}{0},1)\), witnesses a transition whose behaviour on the left boundary is %
\(\binom{0}{0}\) and on the right boundary %
\(1\). Note that, since the graph $N$ has a single vertex, gates are \emph{stateless} components.

The second component is a span
$L=\{ \LatchSet , \LatchReset , \LatchIdle \} \to \{ \LatchA, \LatchNotA \}=R$
that models a set-reset latch.
The diagram below right (\Cref{fig:norandlatch}), again, is a convenient illustration of the span $L \leftarrow D \rightarrow R$. Latches store one bit of information, they are \emph{stateful components};
consequently, their transition graph has two states.
\begin{figure}[H]
  \centering
  \examplesSpanGraph{}
  \caption{A NOR gate and set-reset latch, in $\protect\SpanGraph$.}\label{fig:norandlatch}
\end{figure}

In both transition systems of \Cref{fig:norandlatch} the \emph{interfaces}
 are stateless: indeed, they are determined by a mere set -- the self-loops of a single-vertex graph.
This is a restriction that
occurs rather frequently:
in fact, \emph{transition systems with interfaces} are the arrows of the full
subcategory of $\SpanGraph$ on objects that are single-vertex graphs, which we denote by $\SpanGraphV$.
The objects of \(\SpanGraphV\) represent interfaces, and a morphism \(X \to Y\) encodes a transition system with left interface \(X\) and right interface \(Y\).
Analogously,
the relevant subcategory of $\Cospan(\Graph)$ is $\CospanGraphE$, the full
subcategory on sets, or graphs with an empty set of edges.

\begin{definition*}
  $\SpanGraphV$ is the full subcategory of $\SpanGraph$ with objects the single-vertex graphs.
\end{definition*}

\smallskip
The problem with $\SpanGraphV$ is that it is mysterious from the categorical point of view; the morphisms are graphs, but the boundaries are sets.
\emph{Decorated} and \emph{structured} spans and cospans \cite{fong15,baez19} are
  frameworks that capture such phenomena, which occur frequently when composing network structures. Nevertheless, they do not answer the question of \emph{why} they arise naturally.

\medskip
As stated previously, the main contribution of this paper is the
characterization of $\SpanGraphV$ in terms of a universal property: it is the free \categoryWithFeedback{} over the category of spans of functions. We now state this more formally.

\begin{theorem*}
 The free \categoryWithFeedback{} over \(\Span(\Set)\) is isomorphic to the full subcategory of \(\Span(\Graph)\) given by single-vertex graphs, $\SpanGraphV$.
 That is, there is an isomorphism of categories
 \[\St(\Span(\Set)) \cong \SpanGraphV.\]
\end{theorem*}

Universal constructions, such as the ``state-bootstrapping'' $\St(\bullet)$ construction that yields free categories with feedback, characterize the object of interest up to equivalence, making it \emph{the canonical object} satisfying some properties.
Recall that Abramsky's concern~\cite{abramsky14} is that the lack of consensus about the intrinsic primitives of concurrency risks making the results about any particular model of concurrency too dependent on the specific syntax employed. Characterising a model as satisfying a universal property side-steps this concern.

Given that $\Span(\Set)$, the category of spans of functions, can be considered an \emph{algebra of predicates}~\cite{benabou67,carboni87}, the high level intuition that summarizes our main contribution (\Cref{theorem:spangraphfreefeedback}) can be stated as:
\[\mbox{Algebra of Predicates} + \mbox{Feedback} = \mbox{Algebra of Transition Systems}.\]

We similarly prove (in \Cref{sec:cospangraphfeedback}) that the free \hyperlink{linkcategorywithfeedback}{feedback category} over \(\Cospan(\Set)\) is isomorphic to $\CospanGraphE$, the
full subcategory on discrete graphs of $\Cospan(\Graph)$.

Finally, \Cref{sec:generalizingfeedback} shows how the same framework of \categoriesWithFeedback{} can be extended from transition systems to categories with a structured state space (\Cref{th:statestructured}), such as categories of automata.
As examples, we recover Mealy deterministic finite automata (\Cref{prop:mealydeterministicautomata}) and we introduce span automata (\Cref{def:spanautomata}).

\subsection{Related Work}
This article is an extended version of ``A Canonical Algebra of Open Transition Systems''~\cite{acanonicalalgebra}, presented at the International Conference on Formal Aspects of Component Software (FACS) 2021.
With respect to the conference version, we significantly generalised the framework of \categoriesWithFeedback{}:
\Cref{sec:generalizingfeedback} is completely new material.
At the same time, \Cref{section:preliminaries,section:spangraph} extend the original manuscript adding new proofs (to \cref{lemma:composingstatefulspans,prop:isospans,lemma:spanwelldefined,theorem:spangraphfreefeedback}) and giving a more complete account of the algebra of spans (\Cref{sec:algebraofspans,sec:algebratransition}). In an effort to make the paper more self-contained, we also include a new preliminary \Cref{sec:background-cats}, which summarises the necessary concepts from category theory.

\medskip
$\Span/\Cospan(\Graph)$ has been used for the modeling of concurrent systems \cite{Bruni2011,gianola20a,gianola20b,gianola17,katis97,katis00,sabadini17,Soboci'nski2009a,Sobocinski2010}.
Similar approaches to
compositional modeling of networks have used \emph{decorated} and
\emph{structured cospans} \cite{fong15,baez19}. However, these models have not previously been characterized in terms of a universal property.

In \cite{katis02}, the $\St(\bullet)$ construction (under a different name) is
exhibited as the free \emph{feedback category}. Feedback categories
have been arguably under-appreciated but, at the same time, the $\St(\bullet)$
construction has made multiple appearances as a ``state bootstrapping''
technique across the literature. The $\St(\bullet)$ construction is used to
describe a string diagrammatic syntax for \emph{concurrency theory} in
\cite{bonchi19}; a variant of it had been previously applied in the setting of
\emph{cartesian bicategories} in~\cite{sabadini95}; and it was again rediscovered to describe a \emph{memoryful geometry of interaction} in~\cite{hoshino14}.
However, a coherent account of both feedback categories and their relation with these stateful extensions has not previously appeared.
This motivates our extensive preliminaries in \Cref{section:trace,section:feedback}.

\subsection{Synopsis}
\Cref{sec:background-cats} consists of background material on symmetric monoidal categories and equivalences between them.
\Cref{section:preliminaries} contains preliminary discussions on \tracedMonoidalCategories{} and \hyperlink{linkcategorywithfeedback}{categories with feedback}; it explicitly describes $\St(\bullet)$, the free \categoryWithFeedback{}.
It collects mainly expository material. \Cref{section:spangraph} exhibits a universal property for the $\SpanGraphV$ and $\CospanGraphE$ models of concurrency and \Cref{section:applications} highlights a specific application.
\Cref{sec:generalizingfeedback} extends the framework of \categoriesWithFeedback{} to capture categories of automata.

\section{Preliminaries: Symmetric Monoidal Categories}
\label{sec:background-cats}
\subsection{Theories of Processes}\label{sec:theories-processes}

\paragraph{Resources and processes.} We start by setting up an abstract framework for what it means to describe a theory of processes. A theory of processes contains two kinds of components: some \emph{resource types}, which we name $A, B, C, \dots$; and some \emph{processes}, which we name $f, g, h, \dots$.

Each process $f$ has an associated input resource type (say, $A$); and an associated output resource type (say, $B$). Executing the process $f$ will require some inputs of type $A$ and will produce some outputs of type $B$. We write this situation as $f \colon A \to B$.

Throughout the paper, we make use of \emph{string diagrams}: a formal diagrammatic syntax for theories of processes \cite{joyal96,maclane78}. In a diagram, every ocurrence of a resource type is represented by a laballed wire; every process is represented by a box, with input wires representing its input type on the left, and output wires representing its output type on the right (\Cref{fig:stringdiagramf}).
\begin{figure}
    \centering

\tikzset{every picture/.style={line width=0.75pt}} %

\begin{tikzpicture}[x=0.75pt,y=0.75pt,yscale=-1,xscale=1]
\draw [color={rgb, 255:red, 0; green, 0; blue, 80 }  ,draw opacity=1 ]   (210,80) .. controls (207.56,95.16) and (198.95,96.02) .. (223.08,99.71) ;
\draw [shift={(225,100)}, rotate = 188.48] [color={rgb, 255:red, 0; green, 0; blue, 80 }  ,draw opacity=1 ][line width=0.75]    (10.93,-3.29) .. controls (6.95,-1.4) and (3.31,-0.3) .. (0,0) .. controls (3.31,0.3) and (6.95,1.4) .. (10.93,3.29)   ;
\draw    (180,60) -- (200,60) ;
\draw  [fill={rgb, 255:red, 255; green, 255; blue, 255 }  ,fill opacity=1 ] (200,45) -- (220,45) -- (220,75) -- (200,75) -- cycle ;
\draw    (220,60) -- (240,60) ;
\draw [color={rgb, 255:red, 0; green, 0; blue, 80 }  ,draw opacity=1 ]   (245,70) .. controls (247.27,77.96) and (243.69,79) .. (268.07,79.93) ;
\draw [shift={(270,80)}, rotate = 182.05] [color={rgb, 255:red, 0; green, 0; blue, 80 }  ,draw opacity=1 ][line width=0.75]    (10.93,-3.29) .. controls (6.95,-1.4) and (3.31,-0.3) .. (0,0) .. controls (3.31,0.3) and (6.95,1.4) .. (10.93,3.29)   ;
\draw [color={rgb, 255:red, 0; green, 0; blue, 80 }  ,draw opacity=1 ]   (175,70) .. controls (177.27,77.96) and (173.85,80.06) .. (151.74,80.01) ;
\draw [shift={(150,80)}, rotate = 0.4] [color={rgb, 255:red, 0; green, 0; blue, 80 }  ,draw opacity=1 ][line width=0.75]    (10.93,-3.29) .. controls (6.95,-1.4) and (3.31,-0.3) .. (0,0) .. controls (3.31,0.3) and (6.95,1.4) .. (10.93,3.29)   ;

\draw (210,60) node    {$f$};
\draw (227,100) node [anchor=west] [inner sep=0.75pt]  [color={rgb, 255:red, 0; green, 0; blue, 80 }  ,opacity=1 ]  {$process$};
\draw (272,80) node [anchor=west] [inner sep=0.75pt]  [color={rgb, 255:red, 0; green, 0; blue, 80 }  ,opacity=1 ]  {$output$};
\draw (148,80) node [anchor=east] [inner sep=0.75pt]  [color={rgb, 255:red, 0; green, 0; blue, 80 }  ,opacity=1 ]  {$input$};
\draw (178,60) node [anchor=east] [inner sep=0.75pt]  [font=\scriptsize]  {$A$};
\draw (242,60) node [anchor=west] [inner sep=0.75pt]  [font=\scriptsize]  {$B$};

\end{tikzpicture}
     \caption{String diagram for a process $f \colon A \to B$.}
    \label{fig:stringdiagramf}
\end{figure}

\paragraph{Operations in a theory of processes.}
Theories of processes allow two operations on processes: sequential composition $(\comp)$ and parallel composition $(\tensor)$. The former is depicted as horizontal concatenation of diagrams, the latter as vertical juxtaposition.
\begin{adjustbox}{center=\linewidth,margin=0ex 2ex 0ex 1ex,nofloat=figure}

\tikzset{every picture/.style={line width=0.75pt}} %

\begin{tikzpicture}[x=0.75pt,y=0.75pt,yscale=-1,xscale=1]
\draw    (20,25) -- (35,25) ;
\draw   (35,15) -- (55,15) -- (55,35) -- (35,35) -- cycle ;
\draw    (55,25) -- (70,25) ;
\draw    (115,25) -- (130,25) ;
\draw   (130,15) -- (150,15) -- (150,35) -- (130,35) -- cycle ;
\draw    (150,25) -- (165,25) ;
\draw    (220,25) -- (230,25) ;
\draw   (230,15) -- (250,15) -- (250,35) -- (230,35) -- cycle ;
\draw    (250,25) -- (260,25) ;
\draw   (260,15) -- (280,15) -- (280,35) -- (260,35) -- cycle ;
\draw    (280,25) -- (290,25) ;

\draw (45,25) node    {$f$};
\draw (18,25) node [anchor=east] [inner sep=0.75pt]  [font=\huge]  {$($};
\draw (72,25) node [anchor=west] [inner sep=0.75pt]  [font=\huge]  {$)$};
\draw (93,24.5) node    {$\comp$};
\draw (141,24.5) node    {$g$};
\draw (113,25) node [anchor=east] [inner sep=0.75pt]  [font=\huge]  {$($};
\draw (167,25) node [anchor=west] [inner sep=0.75pt]  [font=\huge]  {$)$};
\draw (192.5,25.5) node    {$=$};%
\draw (240,25) node    {$f$};
\draw (271,24.5) node    {$g$};
\draw (218,25) node [anchor=east] [inner sep=0.75pt]  [font=\huge]  {$($};
\draw (292,25) node [anchor=west] [inner sep=0.75pt]  [font=\huge]  {$)$};

\end{tikzpicture} \end{adjustbox}
\begin{adjustbox}{center=\linewidth,margin=0ex 2ex 0ex 1ex,nofloat=figure}
\tikzset{every picture/.style={line width=0.75pt}} %

\begin{tikzpicture}[x=0.75pt,y=0.75pt,yscale=-1,xscale=1]
\draw    (20,25) -- (35,25) ;
\draw   (35,15) -- (55,15) -- (55,35) -- (35,35) -- cycle ;
\draw    (55,25) -- (70,25) ;
\draw    (115,25) -- (130,25) ;
\draw   (130,15) -- (150,15) -- (150,35) -- (130,35) -- cycle ;
\draw    (150,25) -- (165,25) ;
\draw    (220,15) -- (230,15) ;
\draw   (230,5) -- (250,5) -- (250,25) -- (230,25) -- cycle ;
\draw    (250,15) -- (260,15) ;
\draw   (230,30) -- (250,30) -- (250,50) -- (230,50) -- cycle ;
\draw    (250,40) -- (260,40) ;
\draw    (220,40) -- (230,40) ;

\draw (45,25) node    {$f$};
\draw (18,25) node [anchor=east] [inner sep=0.75pt]  [font=\huge]  {$($};
\draw (72,25) node [anchor=west] [inner sep=0.75pt]  [font=\huge]  {$)$};
\draw (93,24.5) node    {$\otimes $};
\draw (141,24.5) node    {$g$};
\draw (113,25) node [anchor=east] [inner sep=0.75pt]  [font=\huge]  {$($};
\draw (167,25) node [anchor=west] [inner sep=0.75pt]  [font=\huge]  {$)$};
\draw (192.5,25.5) node    {$=$};
\draw (240,15) node    {$f$};
\draw (218,25) node [anchor=east] [inner sep=0.75pt]  [font=\Huge]  {$($};
\draw (262,25) node [anchor=west] [inner sep=0.75pt]  [font=\Huge]  {$)$};
\draw (240,40) node    {$g$};

\end{tikzpicture}
 \end{adjustbox}

\paragraph{Joining resources.} In a theory of processes, resources can be joined. Given a resource type $A$ and a resource type $B$, we can construct the \emph{joint resource type} $A \tensor B$, which puts together resources of type $A$ and type $B$. Resource joining may be implemented in diverse ways, depending on the theory of processes. However, it must satisfy some basic axioms:
\begin{itemize}[label=$\bullet$]
    \item joining three process resource types together can be done in two ways; these should coincide, 
    \begin{equation}
        A \tensor (B \tensor C) = (A \tensor B) \tensor C,
    \end{equation}
    \begin{adjustbox}{       
        center=\linewidth,
        margin=0ex 3ex 0ex 1ex,
        nofloat=figure}
\tikzset{every picture/.style={line width=0.85pt}} %
\begin{tikzpicture}[x=0.75pt,y=0.75pt,yscale=-1,xscale=1]
\draw    (25,17) -- (85,17) ;
\draw    (25,37) -- (85,37) ;
\draw    (25,47) -- (85,47) ;
\draw    (110,17) -- (170,17) ;
\draw    (110,27) -- (170,27) ;
\draw    (110,47) -- (170,47) ;
\draw (86,25.4) node [anchor=north west][inner sep=0.75pt]    {$=$};
\draw (23,17) node [anchor=east] [inner sep=0.75pt]  [font=\scriptsize]  {$A$};
\draw (23,37) node [anchor=east] [inner sep=0.75pt]  [font=\scriptsize]  {$B$};
\draw (23,47) node [anchor=east] [inner sep=0.75pt]  [font=\scriptsize]  {$C$};
\draw (172,17) node [anchor=west] [inner sep=0.75pt]  [font=\scriptsize]  {$A$};
\draw (172,27) node [anchor=west] [inner sep=0.75pt]  [font=\scriptsize]  {$B$};
\draw (172,47) node [anchor=west] [inner sep=0.75pt]  [font=\scriptsize]  {$C$};
\end{tikzpicture}
     \end{adjustbox}
    \item there must exist a resource type representing the absence of resources, which we call the \emph{unit resource type} $I$; it must be neutral with respect to process joining
    \begin{equation}
        A \tensor I = A = A \tensor I.
    \end{equation}
    \begin{adjustbox}{center=\linewidth,margin=0ex 2ex 0ex 2ex,nofloat=figure}
\tikzset{every picture/.style={line width=0.85pt}} %

\begin{tikzpicture}[x=0.75pt,y=0.75pt,yscale=-1,xscale=1]
\draw    (25,27) -- (85,27) ;
\draw    (110,22) -- (170,22) ;
\draw    (195,17) -- (255,17) ;
\draw [color={rgb, 255:red, 0; green, 0; blue, 80 }  ,draw opacity=1 ] [dash pattern={on 4.5pt off 4.5pt}]  (195,27) -- (255,27) ;
\draw [color={rgb, 255:red, 0; green, 0; blue, 80 }  ,draw opacity=1 ] [dash pattern={on 4.5pt off 4.5pt}]  (25,17) -- (85,17) ;
\draw (23,27) node [anchor=east] [inner sep=0.75pt]  [font=\scriptsize]  {$A$};
\draw (23,17) node [anchor=east] [inner sep=0.75pt]  [font=\scriptsize,color={rgb, 255:red, 0; green, 0; blue, 80 }  ,opacity=1 ]  {$I$};
\draw (86,14.4) node [anchor=north west][inner sep=0.75pt]    {$=$};
\draw (171,15.4) node [anchor=north west][inner sep=0.75pt]    {$=$};
\draw (257,17) node [anchor=west] [inner sep=0.75pt]  [font=\scriptsize]  {$A$};
\draw (257,27) node [anchor=west] [inner sep=0.75pt]  [font=\scriptsize,color={rgb, 255:red, 0; green, 0; blue, 80 }  ,opacity=1 ]  {$I$};

\end{tikzpicture}     \end{adjustbox}
\end{itemize}

\paragraph{Sequential composition.} In a theory of processes, we can compose processes in two different ways. The first is sequential composition: given two processes such that the output type of the first coincides with the input type of the second, say $f \colon A \to B$ and $g \colon B \to C$, their \emph{sequential composition} is the process $(f \comp g) \colon A \to C$ that results from executing $f$ and using its output to execute $g$.

Composing may mean different things in different process theories, but it must always satisfy the following axioms:
\begin{itemize}[label=$\bullet$]
    \item sequencing together three processes $f \colon A \to B$, $g \colon B \to C$ and $h \colon C \to D$ can be done in two different ways, these should coincide,
    \begin{equation}
        (f \comp g) \comp h = f \comp (g \comp h);
    \end{equation}
    \begin{adjustbox}{center=\linewidth,margin=0ex 2ex 0ex 2ex,nofloat=figure}
\tikzset{every picture/.style={line width=0.85pt}} %

\begin{tikzpicture}[x=0.75pt,y=0.75pt,yscale=-1,xscale=1]
\draw    (30,30) -- (40,30) ;
\draw   (40,20) -- (60,20) -- (60,40) -- (40,40) -- cycle ;
\draw    (60,30) -- (70,30) ;
\draw   (70,20) -- (90,20) -- (90,40) -- (70,40) -- cycle ;
\draw    (90,30) -- (110,30) ;
\draw   (110,20) -- (130,20) -- (130,40) -- (110,40) -- cycle ;
\draw    (130,30) -- (140,30) ;
\draw    (165,30) -- (175,30) ;
\draw   (175,20) -- (195,20) -- (195,40) -- (175,40) -- cycle ;
\draw    (235,30) -- (245,30) ;
\draw   (215,20) -- (235,20) -- (235,40) -- (215,40) -- cycle ;
\draw    (195,30) -- (215,30) ;
\draw   (245,20) -- (265,20) -- (265,40) -- (245,40) -- cycle ;
\draw    (265,30) -- (275,30) ;

\draw (50,30) node    {$f$};
\draw (80,30) node    {$g$};
\draw (120,30) node    {$h$};
\draw (185,30) node    {$f$};
\draw (225,30) node    {$g$};
\draw (255,30) node    {$h$};
\draw (141,23.4) node [anchor=north west][inner sep=0.75pt]    {$=$};

\end{tikzpicture}         \end{adjustbox}
    \item there must exist a process representing ``doing nothing'' with a resource $A$ that we write as $\im_A$ -- the \emph{identity} transformation -- which must be neutral with respect to sequential composition,
    \begin{equation}
        \im_A \comp f = f = f \comp \im_B.
    \end{equation}
    \begin{adjustbox}{center=\linewidth,margin=0ex 2ex 0ex 2ex,nofloat=figure}

\tikzset{every picture/.style={line width=0.85pt}} %

\begin{tikzpicture}[x=0.75pt,y=0.75pt,yscale=-1,xscale=1]
\draw    (245,53) -- (255,53) ;
\draw   (255,43) -- (275,43) -- (275,63) -- (255,63) -- cycle ;
\draw    (275,53) -- (305,53) ;
\draw    (150,53) -- (160,53) ;
\draw    (90,53) -- (100,53) ;
\draw   (70,43) -- (90,43) -- (90,63) -- (70,63) -- cycle ;
\draw    (40,53) -- (70,53) ;
\draw   (160,43) -- (180,43) -- (180,63) -- (160,63) -- cycle ;
\draw    (180,53) -- (190,53) ;

\draw (265,53) node    {$f$};
\draw (116,46.4) node [anchor=north west][inner sep=0.75pt]    {$=$};
\draw (80,53) node    {$f$};
\draw (206,46.4) node [anchor=north west][inner sep=0.75pt]    {$=$};
\draw (170,53) node    {$f$};
\draw (38,53) node [anchor=east] [inner sep=0.75pt]  [font=\scriptsize]  {$A$};
\draw (307,53) node [anchor=west] [inner sep=0.75pt]  [font=\scriptsize]  {$B$};
\draw (102,53) node [anchor=west] [inner sep=0.75pt]  [font=\scriptsize]  {$B$};
\draw (148,53) node [anchor=east] [inner sep=0.75pt]  [font=\scriptsize]  {$A$};
\draw (192,53) node [anchor=west] [inner sep=0.75pt]  [font=\scriptsize]  {$B$};
\draw (243,53) node [anchor=east] [inner sep=0.75pt]  [font=\scriptsize]  {$A$};

\end{tikzpicture}     \end{adjustbox}
\end{itemize}
We say that a process $f \colon A \to B$ is \emph{reversible} if it has a reverse counterpart, $\inverse{f} \colon B \to A$, such that executing one after the other is the same as having done nothing, $f \comp \inverse{f} = \im_A$ and $\inverse{f} \comp f = \im_B$. This is usually called an \emph{isomorphism}. In this situation, we say that $A$ and $B$ are \emph{isomorphic}, and we write that as $A \cong B$.

\paragraph{Parallel composition.} The second way of composing two processes is to do so in parallel. Given any two processes $f \colon A \to B$ and $f' \colon A' \to B'$, their parallel composition is a process $(f \tensor f') \colon A \tensor A' \to B \tensor B'$ that results from jointly executing both processes over the joint input resource type, so as to produce the joint output resource type.

The implementation of parallel composition will usually be related to the implementation of resource joining in the same theory. It must satisfy the following axioms:
\begin{itemize}[label=$\bullet$]
    \item composing three processes in parallel can be done in two ways; these should coincide,
    \begin{equation}
        f \tensor (g \tensor h) = (f \tensor g) \tensor h;
    \end{equation}
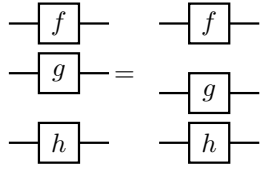
    \begin{adjustbox}{center=\linewidth,margin=0ex 2ex 0ex 1ex,nofloat=figure}
\tikzset{every picture/.style={line width=0.85pt}} %

\begin{tikzpicture}[x=0.75pt,y=0.75pt,yscale=-1,xscale=1]
\draw    (15,20) -- (30,20) ;
\draw   (30,10) -- (50,10) -- (50,30) -- (30,30) -- cycle ;
\draw    (50,20) -- (65,20) ;
\draw    (15,45) -- (30,45) ;
\draw   (30,35) -- (50,35) -- (50,55) -- (30,55) -- cycle ;
\draw    (50,45) -- (65,45) ;
\draw    (15,80) -- (30,80) ;
\draw   (30,70) -- (50,70) -- (50,90) -- (30,90) -- cycle ;
\draw    (50,80) -- (65,80) ;
\draw    (90,20) -- (105,20) ;
\draw   (105,10) -- (125,10) -- (125,30) -- (105,30) -- cycle ;
\draw    (125,20) -- (140,20) ;
\draw    (90,55) -- (105,55) ;
\draw   (105,45) -- (125,45) -- (125,65) -- (105,65) -- cycle ;
\draw    (125,55) -- (140,55) ;
\draw    (90,80) -- (105,80) ;
\draw   (105,70) -- (125,70) -- (125,90) -- (105,90) -- cycle ;
\draw    (125,80) -- (140,80) ;

\draw (40,20) node    {$f$};
\draw (40,45) node    {$g$};
\draw (40,80) node    {$h$};
\draw (115,20) node    {$f$};
\draw (115,55) node    {$g$};
\draw (115,80) node    {$h$};
\draw (66,42.4) node [anchor=north west][inner sep=0.75pt]    {$=$};

\end{tikzpicture}
     \end{adjustbox}
    \item doing nothing with no resources should be the unit for parallel composition; the identity transformation on the unit resource type $I$ must satisfy
    \begin{equation}
        f \tensor \im_I = f = \im_I \tensor f;
    \end{equation}
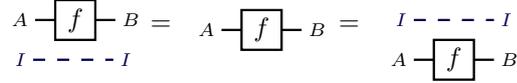
    \begin{adjustbox}{center=\linewidth,margin=0ex 2ex 0ex 1ex,nofloat=figure}
\tikzset{every picture/.style={line width=0.85pt}} %

\begin{tikzpicture}[x=0.75pt,y=0.75pt,yscale=-1,xscale=1]
\draw    (51,33) -- (61,33) ;
\draw   (61,23) -- (81,23) -- (81,43) -- (61,43) -- cycle ;
\draw    (81,33) -- (91,33) ;
\draw [color={rgb, 255:red, 0; green, 0; blue, 80 }  ,draw opacity=1 ] [dash pattern={on 4.5pt off 4.5pt}]  (50,53) -- (90,53) ;
\draw    (144,38) -- (154,38) ;
\draw   (154,28) -- (174,28) -- (174,48) -- (154,48) -- cycle ;
\draw    (174,38) -- (184,38) ;
\draw    (240,53) -- (250,53) ;
\draw   (250,43) -- (270,43) -- (270,63) -- (250,63) -- cycle ;
\draw    (270,53) -- (280,53) ;
\draw [color={rgb, 255:red, 0; green, 0; blue, 80 }  ,draw opacity=1 ] [dash pattern={on 4.5pt off 4.5pt}]  (240,33) -- (280,33) ;

\draw (71,33) node    {$f$};
\draw (49,33) node [anchor=east] [inner sep=0.75pt]  [font=\scriptsize]  {$A$};
\draw (93,33) node [anchor=west] [inner sep=0.75pt]  [font=\scriptsize]  {$B$};
\draw (48,53) node [anchor=east] [inner sep=0.75pt]  [font=\scriptsize,color={rgb, 255:red, 0; green, 0; blue, 80 }  ,opacity=1 ]  {$I$};
\draw (92,53) node [anchor=west] [inner sep=0.75pt]  [font=\scriptsize,color={rgb, 255:red, 0; green, 0; blue, 80 }  ,opacity=1 ]  {$I$};
\draw (107,31.4) node [anchor=north west][inner sep=0.75pt]    {$=$};
\draw (164,38) node    {$f$};
\draw (142,38) node [anchor=east] [inner sep=0.75pt]  [font=\scriptsize]  {$A$};
\draw (186,38) node [anchor=west] [inner sep=0.75pt]  [font=\scriptsize]  {$B$};
\draw (201,31.4) node [anchor=north west][inner sep=0.75pt]    {$=$};
\draw (260,53) node    {$f$};
\draw (238,53) node [anchor=east] [inner sep=0.75pt]  [font=\scriptsize]  {$A$};
\draw (282,53) node [anchor=west] [inner sep=0.75pt]  [font=\scriptsize]  {$B$};
\draw (238,33) node [anchor=east] [inner sep=0.75pt]  [font=\scriptsize,color={rgb, 255:red, 0; green, 0; blue, 80 }  ,opacity=1 ]  {$I$};
\draw (282,33) node [anchor=west] [inner sep=0.75pt]  [font=\scriptsize,color={rgb, 255:red, 0; green, 0; blue, 80 }  ,opacity=1 ]  {$I$};

\end{tikzpicture}
     \end{adjustbox}
    \item executing two processes in parallel and then other two processes in parallel must yield the same result as executing in parallel the sequential compositions of both pairs,
    \begin{equation}
        (f \tensor g) \comp (h \tensor k) = (f \comp h) \tensor (g \comp k).
    \end{equation}
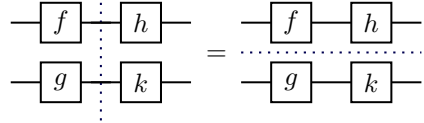
    \begin{adjustbox}{center=\linewidth,margin=0ex 2ex 0ex 1ex,nofloat=figure}
\tikzset{every picture/.style={line width=0.75pt}} %

\begin{tikzpicture}[x=0.75pt,y=0.75pt,yscale=-1,xscale=1]
\draw    (260,47) -- (275,47) ;
\draw   (275,37) -- (295,37) -- (295,57) -- (275,57) -- cycle ;
\draw    (295,47) -- (310,47) ;
\draw    (260,77) -- (275,77) ;
\draw   (275,67) -- (295,67) -- (295,87) -- (275,87) -- cycle ;
\draw    (295,77) -- (310,77) ;
\draw [color={rgb, 255:red, 0; green, 0; blue, 80 }  ,draw opacity=1 ] [dash pattern={on 0.84pt off 2.51pt}]  (305,37) -- (305,97) ;
\draw    (300,47) -- (315,47) ;
\draw   (315,37) -- (335,37) -- (335,57) -- (315,57) -- cycle ;
\draw    (335,47) -- (350,47) ;
\draw    (300,77) -- (315,77) ;
\draw   (315,67) -- (335,67) -- (335,87) -- (315,87) -- cycle ;
\draw    (335,77) -- (350,77) ;
\draw    (375,47) -- (390,47) ;
\draw   (390,37) -- (410,37) -- (410,57) -- (390,57) -- cycle ;
\draw    (375,77) -- (390,77) ;
\draw   (390,67) -- (410,67) -- (410,87) -- (390,87) -- cycle ;
\draw    (410,77) -- (430,77) ;
\draw    (410,47) -- (430,47) ;
\draw   (430,37) -- (450,37) -- (450,57) -- (430,57) -- cycle ;
\draw    (450,47) -- (465,47) ;
\draw   (430,67) -- (450,67) -- (450,87) -- (430,87) -- cycle ;
\draw    (450,77) -- (465,77) ;
\draw [color={rgb, 255:red, 0; green, 0; blue, 80 }  ,draw opacity=1 ] [dash pattern={on 0.84pt off 2.51pt}]  (375,62) -- (465,62) ;

\draw (285,47) node    {$f$};
\draw (285,77) node    {$g$};
\draw (325,47) node    {$h$};
\draw (325,77) node    {$k$};
\draw (400,47) node    {$f$};
\draw (400,77) node    {$g$};
\draw (440,47) node    {$h$};
\draw (440,77) node    {$k$};
\draw (356,59.4) node [anchor=north west][inner sep=0.75pt]    {$=$};

\end{tikzpicture}
     \end{adjustbox}
\end{itemize}

\paragraph{Swapping.} Finally, we want to be able to route resources to each specific process. Any theory of processes, given any two resource types $A$ and $B$, must contain a process $\sigma_{A,B} \colon A \tensor B \to B \tensor A$. This process is called the \emph{swap}, which only permutes the order in which resources are organized. It must satisfy the following axioms.
\begin{itemize}[label=$\bullet$]
    \item Swapping twice is the same as swapping once with a joint type,
    \begin{align}
      \sigma_{A, B \tensor C} &= (\sigma_{A,B} \comp \im_C) \comp (\im_B \tensor \sigma_{A,C}); \\
      \sigma_{A \tensor B,C} &=
      (\im_A \comp \sigma_{B,C}) \comp (\sigma_{A,C} \tensor \im_B).
    \end{align}
    \begin{adjustbox}{center=\linewidth,margin=0ex 2ex 0ex 1ex,nofloat=figure}
\tikzset{every picture/.style={line width=0.85pt}} %

\begin{tikzpicture}[x=0.75pt,y=0.75pt,yscale=-1,xscale=1]
\draw    (25,10) .. controls (41,10.4) and (40,39.8) .. (55,40) ;
\draw    (10,10) -- (25,10) ;
\draw    (55,40) -- (70,40) ;
\draw    (10,30) -- (25,30) ;
\draw    (10,40) -- (25,40) ;
\draw    (25,30) .. controls (41,30.4) and (40,9.8) .. (55,10) ;
\draw    (25,40) .. controls (41,40.4) and (40,19.8) .. (55,20) ;
\draw    (55,20) -- (70,20) ;
\draw    (55,10) -- (70,10) ;
\draw    (95,10) -- (105,10) ;
\draw    (95,25) -- (105,25) ;
\draw    (95,40) -- (120,40) ;
\draw    (105,10) .. controls (115.6,10.2) and (109.6,25.4) .. (120,25) ;
\draw    (105,25) .. controls (115.6,25.2) and (109.6,10.4) .. (120,10) ;
\draw    (120,25) .. controls (130.6,25.2) and (124.6,40.4) .. (135,40) ;
\draw    (120,40) .. controls (130.6,40.2) and (124.6,25.4) .. (135,25) ;
\draw    (120,10) -- (145,10) ;
\draw    (135,25) -- (145,25) ;
\draw    (135,40) -- (145,40) ;

\draw (76,22.4) node [anchor=north west][inner sep=0.75pt]    {$=$};

\end{tikzpicture}
     \end{adjustbox}
    \begin{adjustbox}{center=\linewidth,margin=0ex 2ex 0ex 1ex,nofloat=figure}
\tikzset{every picture/.style={line width=0.85pt}} %

\begin{tikzpicture}[x=0.75pt,y=0.75pt,yscale=-1,xscale=1]
\draw    (25,40) .. controls (41,39.6) and (40,10.21) .. (55,10.01) ;
\draw    (10,40) -- (25,40) ;
\draw    (55,10.01) -- (70,10.01) ;
\draw    (10,20.01) -- (25,20) ;
\draw    (10,10.01) -- (25,10.01) ;
\draw    (25,20) .. controls (41,19.6) and (40,40.2) .. (55,40) ;
\draw    (25,10.01) .. controls (41,9.61) and (40,30.2) .. (55,30) ;
\draw    (55,30) -- (70,30) ;
\draw    (55,40) -- (70,40) ;
\draw    (95,40) -- (105,40) ;
\draw    (95,25) -- (105,25) ;
\draw    (95,10.01) -- (120,10.01) ;
\draw    (105,40) .. controls (115.6,39.8) and (109.6,24.6) .. (120,25) ;
\draw    (105,25) .. controls (115.6,24.8) and (109.6,39.6) .. (120,40) ;
\draw    (120,25) .. controls (130.6,24.8) and (124.6,9.61) .. (135,10.01) ;
\draw    (120,10.01) .. controls (130.6,9.81) and (124.6,24.6) .. (135,25) ;
\draw    (120,40) -- (145,40) ;
\draw    (135,25) -- (145,25) ;
\draw    (135,10.01) -- (145,10.01) ;

\draw (90,30) node [anchor=north east][inner sep=0.75pt]  [rotate=-180,xscale=-1]  {$=$};

\end{tikzpicture}
\     \end{adjustbox}
    \item Swapping two process inputs is the same as swapping the executing place and swapping the output.
    \begin{equation}
        (f \tensor g) \comp \sigma_{B,B'} =
        \sigma_{A,A'} \comp (g \tensor f).
    \end{equation}
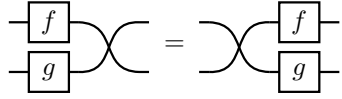
    \begin{adjustbox}{center=\linewidth,margin=0ex 2ex 0ex 1ex,nofloat=figure}

\tikzset{every picture/.style={line width=0.85pt}} %

\begin{tikzpicture}[x=0.75pt,y=0.75pt,yscale=-1,xscale=1]
\draw    (10,20) -- (20,20) ;
\draw   (20,10) -- (40,10) -- (40,30) -- (20,30) -- cycle ;
\draw    (40,20) -- (45,20) ;
\draw    (10,45) -- (20,45) ;
\draw   (20,35) -- (40,35) -- (40,55) -- (20,55) -- cycle ;
\draw    (45,20) .. controls (65.4,20.8) and (54.6,44.4) .. (75,45) ;
\draw    (45,45) .. controls (65.4,45.8) and (54.6,19.4) .. (75,20) ;
\draw    (40,45) -- (45,45) ;
\draw    (75,20) -- (80,20) ;
\draw    (75,45) -- (80,45) ;
\draw    (165,20) -- (175,20) ;
\draw   (145,10) -- (165,10) -- (165,30) -- (145,30) -- cycle ;
\draw    (105,20) -- (110,20) ;
\draw    (165,45) -- (175,45) ;
\draw   (145,35) -- (165,35) -- (165,55) -- (145,55) -- cycle ;
\draw    (110,20) .. controls (130.4,20.8) and (119.6,44.4) .. (140,45) ;
\draw    (110,45) .. controls (130.4,45.8) and (119.6,19.4) .. (140,20) ;
\draw    (105,45) -- (110,45) ;
\draw    (140,20) -- (145,20) ;
\draw    (140,45) -- (145,45) ;

\draw (30,20) node    {$f$};
\draw (30,45) node    {$g$};
\draw (155,20) node    {$f$};
\draw (155,45) node    {$g$};
\draw (86,27.4) node [anchor=north west][inner sep=0.75pt]    {$=$};

\end{tikzpicture}
     \end{adjustbox}
    \item Swapping and swapping again is the same as doing nothing.
    \begin{equation}
        \sigma_{A,B} \comp \sigma_{B,A} = \im_{A \tensor B}.
    \end{equation}
    \begin{adjustbox}{center=\linewidth,margin=0ex 2ex 0ex 1ex,nofloat=figure}

\tikzset{every picture/.style={line width=0.85pt}} %

\begin{tikzpicture}[x=0.75pt,y=0.75pt,yscale=-1,xscale=1]
\draw    (10,15.01) .. controls (29.2,14.83) and (19.6,29.64) .. (40,30) ;
\draw    (10,30) .. controls (30.4,30.48) and (19.6,14.65) .. (40,15.01) ;
\draw    (40,15.03) .. controls (60.4,15.5) and (49.6,29.66) .. (70,30.02) ;
\draw    (40,30.01) .. controls (60.4,30.49) and (49.6,14.67) .. (70,15.03) ;
\draw    (95,15.03) -- (135,15.03) ;
\draw    (95,30.03) -- (135,30.03) ;

\draw (76,17.43) node [anchor=north west][inner sep=0.75pt]    {$=$};

\end{tikzpicture}
     \end{adjustbox}
\end{itemize}

\paragraph{Symmetric monoidal categories.} The algebraic structures that capture this notion of process theory are ``symmetric monoidal categories'' \cite{maclane78}. The resource types are usually called \emph{objects}, while the processes are usually called \emph{morphisms}. Reversible processes are called \emph{isomorphisms}.

\begin{definition}
    A symmetric monoidal category \cite{maclane78} is a tuple
    \[\catC = (\catC_\mathrm{obj}, \catC_\mathrm{mor}, (\comp), \im, (\tensor)_\mathrm{obj}, (\tensor)_\mathrm{mor}, I, \sigma),\]
    specifying a set of objects, or resource types, $\catC_\mathrm{obj}$; a set of morphisms, or processes, $\catC_\mathrm{mor}$; a composition operation; a family of identity morphisms; a tensor operation on objects and morphisms; a unit object and a family of swapping morphisms; satisfying all of the axioms of this section (1-11), possibly up to reversible coherence isomorphisms of the form,
    \[\begin{aligned}
        \alpha_{A,B,C} \colon & (A \tensor B) \tensor C \to A \tensor (B \tensor C), \\
        \lambda_A \colon & I \tensor A \to A, \mbox{ and } \\
        \rho_A \colon & A \tensor I \to A.
    \end{aligned}\]
    Coherence isomorphisms must commute with all suitably typed processes and must satisfy all possible formal equations between them.
    We usually denote by $\catC(A,B)$ the set of morphisms from $A$ to $B$.
\end{definition}

Note that we do allow the axioms to be satisfied \emph{up to a reversible coherence isomorphism}.
For an example, consider the theory of pure functions between sets joined by the cartesian product. It is not true that, given three sets $A$, $B$ and $C$, the following two sets are \emph{equal}, $A \times (B \times C) \cong (A \times B) \times C$; they are merely in a one-to-one correspondence.
A symmetric monoidal category is \emph{strict} only if these reversible transformations are identities.
It was proven by MacLane (his Coherence Theorem, \Cref{theorem:coherence} \cite{maclane78}) that the axioms (1-11) are valid for both strict and non-strict monoidal categories.

\begin{example}
The paradigmatic theory of processes uses mathematical sets as types and functions as processes. We can check that the following functions, with the cartesian product, satisfy the axioms (1-11), thus forming a symmetric monoidal category. 
\[\begin{aligned}
    \mathbf{Set} = (\mathsf{Sets}, & \mathsf{Functions}, (\circ), \im, \times, \langle \bullet, \bullet \rangle, 1, (a,(b,c)) \mapsto ((a,b),c), \\ &  (a,\ast) \mapsto a, (\ast, a) \mapsto a , (a,b) \mapsto (b,a)).
\end{aligned}\]
\end{example}

\begin{example}\label{definition:matcategory}
The theory of linear transformations uses dimensions (natural numbers) as types and matrices over the real numbers as processes. We can check that matrices, with the direct sum, satisfy the axioms (1-11), thus forming a symmetric monoidal category.
\[\begin{aligned}
    \mathbf{Mat} = (\mathbb{N}, & \mathsf{Matrices}, (\cdot), (+), \oplus, 0, \mathbf{I}, \mathbf{I}, \mathbf{I}, \mathbf{I}, \mathbf{S}),
\end{aligned}\]
where $\mathbf{I}$ is the identity matrix and $\mathbf{S}$ is the permutation matrix,
\[ \mathbf{I}_n = \begin{psmallmatrix} 1 & \dots & 0 \\ \vdots & \ddots n & \vdots \\ 0 & \dots & 1 \end{psmallmatrix} ; \qquad
\mathbf{S}_{n,m} = \begin{psmallmatrix} 0 & \dots & 0 & 1 & \dots & 0 \\ \vdots & \ddots  & \vdots & \vdots & \ddots n & \vdots \\ 0 & \dots & 0 & 0 & \dots & 1 \\ 1 & \dots & 0 & 0 & \dots & 0 \\ \vdots & \ddots m & \vdots & \vdots & \ddots & \vdots \\ 0 & \dots & 1 & 0 & \dots & 0 \end{psmallmatrix}.\]
\end{example}

\begin{example}
\label{ex:finitesetplus}
It can happen that two theories of processes share the same elements, but differ on how they are combined. The theory of choice in finite sets uses again functions, but instead of the cartesian product, it uses the disjoint union. We can check that the following functions satisfy again the axioms (1-11).
\[\begin{aligned}
    \mathbf{FinSet} = (\mathsf{FinSets}, & \mathsf{Functions}, (\circ), \im, (+), [ \bullet, \bullet ], 0, (a|(b|c)) \mapsto ((a|b)|c), \\ &  (a|\emptyset) \mapsto a, (\emptyset | a) \mapsto a , (a|b) \mapsto (b|a)).
\end{aligned}\]
\end{example}

When designing software, the advantage of an algebraic structure such as monoidal categories is reusability: we can encapsulate the operations of our theory of processes into a separate module, and we can abstractly work with them without knowing the particulars of the theory of processes at hand. The axioms (1-11) are straightforward to check for most theories of processes -- even if we will not take the time to do so in this text -- but they are a powerful abstraction: once the axioms are satisfied, we can start reasoning with string diagrams.

\subsection{Monoidal Equivalence}

In this final preliminary section, we recall what it means to have a transformation between monoidal categories (symmetric strong monoidal functor, \Cref{def:symmetricstrongmonoidal}), what it means to have two equivalent \monoidalCategories{} (\monoidalEquivalence{}, \Cref{def:monoidalequivalence}) and the statement of the Coherence Theorem: every \monoidalCategory{} is equivalent to a strict one (\Cref{theorem:coherence}).

\paragraph{Monoidal functors.}
Every time we consider an algebraic structure, it is natural to also consider what is a good notion of transformation between two such algebraic structures. A transformation of algebraic structures should preserve the key ingredients of the algebraic construction. In the case of \symmetricMonoidalCategories{}, these transformations are called \emph{monoidal functors}, and they preserve the operation of composition.

\begin{definition}
    \label{def:symmetricstrongmonoidal}
    A \defining{linkmonoidalfunctor}{symmetric strong monoidal functor} between two \symmetricMonoidalCategories{} with coherence isomorphisms
    \[\begin{aligned}
        \catC &= (\catC_\mathrm{obj}, \catC_\mathrm{mor}, (\comp), \im, (\tensor)_\mathrm{obj}, (\tensor)_\mathrm{mor}, I, \alpha^\catC, \lambda^\catC, \rho^\catC, \sigma^\catC), \mbox{ and } \\
        \catD &= (\catD_\mathrm{obj}, \catD_\mathrm{mor}, (\comp), \im, (\tensor)_\mathrm{obj}, (\tensor)_\mathrm{mor}, I, \alpha^\catD, \lambda^\catD, \rho^\catD, \sigma^\catD)
    \end{aligned}\]
    is a tuple $\mathbf{F} = (F_\mathrm{obj}, F_\mathrm{mor},\phi,\varphi)$, consisting of 
    \begin{itemize}[label=$\bullet$]
        \item a function that assigns objects of the first category to objects of the second category, $F_\mathrm{obj} \colon \catC_\mathrm{obj} \to \catD_\mathrm{obj}$, 
        \item and a function that assigns morphisms of the first category to morphisms of the second category, $F_\mathrm{mor} \colon \catC_\mathrm{mor} \to \catD_\mathrm{mor}$.
        \item a coherence isomorphism $\phi_{A,B} \colon FA \tensor FB \to F(A \tensor B)$,
        \item and a coherence isomorphism $\varphi \colon J \to FI$.
    \end{itemize}
    Traditionally, functions both on objects, $F_\mathrm{obj}$ and morphisms, $F_\mathrm{mor}$ are denoted by $F$.
    The functor must be such that every morphism $f \colon A \to B$ is assigned a morphism $F(f) \colon FA \to FB$, whose source and target are the images of the original source and target. Moreover, it must satisfy the following axioms,
    \begin{itemize}[label=$\bullet$]
        \item compositions must be preserved, $F(f \comp g) = F(f) \comp F(g)$,
        \item identities must be preserved, $F(\im_A) = \im_{FA}$,
        \item tensoring must be transported by the natural transformations, meaning that
        \[F(f \tensor g) = \mu \comp (F(f) \tensor F(g)) \comp \inverse{\mu},\]
        \item associators, unitors and swaps must be transported by the natural transformations, meaning that
        \[\begin{aligned}
            F(\alpha^\catC) &= \inverse{\mu} \comp (\inverse{\mu} \tensor \im) \comp \alpha^\catD \comp (\im \tensor \mu) \comp \mu, \\
            F(\lambda^\catC) &= \inverse{\mu} \comp (\inverse{\varphi} \tensor \im) \comp \lambda^\catD, \\
            F(\rho^\catC) &= \inverse{\mu} \comp (\im \tensor \inverse{\varphi}) \comp \rho^\catD, \\
            F(\sigma^\catC) &= \inverse{\mu} \comp \sigma^\catD \comp \mu.
        \end{aligned}\]
    \end{itemize}
\end{definition}

\begin{example}
    \label[]{ex:finsetToMat}
    For instance, there is a strong monoidal functor translating from the theory of choice in finite sets, $\mathbf{FinSet}_{+}$ (\Cref{ex:finitesetplus}), to the theory of linear transformations $\mathbf{Mat}$ (\Cref{definition:matcategory}) that sends the finite sets $A = \{a_0,\dots,a_{n-1}\}$ and $B = \{b_0,\dots,b_{m-1}\}$ to their cardinalities, $n$ and $m$; and each function $f \colon A \to B$ to the matrix $F_{ij} \colon n \to m$ that contains a $1$ on the entry $F_{ij}$ when $f(a_i) = b_j$, and contains a $0$ otherwise.
\end{example}

\begin{definition}
    \label{def:monoidalequivalence}
    \defining{linkisomorphismofcategories}{}
    A \defining{linkmonoidalequivalence}{monoidal equivalence} of categories is a symmetric strong \monoidalFunctor{} $F \colon \catC \to \catD$ that is
    \begin{enumerate}
        \item \emph{essentially surjective} on objects, meaning that for each $X \in \catD_\mathrm{obj}$, there exists $A \in \catC_\mathrm{obj}$ such that $F(A) \cong X$;
        \item \emph{essentially injective} on objects, meaning that $F(A) \cong F(B)$ implies $A \cong B$; it can be proven that every monoidal functor is essentially injective, so this condition, though conceptually important, is superfluous; 
        \item \emph{surjective on morphisms}, or \emph{full}, meaning that for each $g \colon FA \to FB$ there exists some $f \colon A \to B$ such that $F(f) = g$;
        \item \emph{injective on morphisms}, or \emph{faithful}, meaning that given any two morphisms $f \colon A \to B$ and $g \colon A \to B$ such that $F(f) = F(g)$, it holds that $f = g$.
    \end{enumerate} 
    In this situation, we say that $\catC$ and $\catD$ are \emph{equivalent}, and we write that as $\catC \cong \catD$. Moreover, when the \monoidalFunctor{} is injective and surjective on objects, we say that $\catC$ and $\catD$ are \emph{isomorphic}.
\end{definition}

\begin{theorem}[Coherence theorem, {{\cite[Theorem 2.1, Chapter VII]{maclane78}}}]%
    \label{theorem:coherence}
    Every monoidal category is monoidally equivalent to a strict monoidal category.
  \end{theorem}
  
Let us comment further on how we use the coherence theorem. Each time we have a morphism $f \colon A \to B$ in a monoidal category, we have a corresponding morphism $A \to B$ in its strictification. 
This morphism can be lifted to the original category to uniquely produce, say, a morphism $(\lambda_{A} \comp f \comp \inverse{\lambda_{B}}) \colon I \otimes A \to I \otimes B$. Each time the source and the target are clearly determined, we simply write $f$ again for this new morphism.

The reason to avoid this explicit notation on our definitions and proofs is that it would quickly become verbose and distractive.
Equations seem conceptually easier to understand when written assuming the coherence theorem -- and they become even clearer when drawn as string diagrams, which implicitly hide these bureaucratic isomorphisms.
In fact, in the work of Katis, Sabadini and Walters~\cite{katis02}, strictness is assumed from the start for the sake of readibility, even though---as argued above---it is not a necessary assumption.

\Cref{theorem:coherence} and \Cref{sec:theories-processes} can be summarized by the slogan:
\begin{quote}
  \emph{``Any theory of processes satisfying the axioms of symmetric monoidal categories (1-11) can be reasoned about using string diagrams''}.
\end{quote}

 \newpage

\section{Feedback Categories}
\label{section:preliminaries}
In this section we recall \hyperlink{linkcategorywithfeedback}{feedback categories}, originally introduced in~\cite{katis02}, and contrast them with the stronger notion of \emph{traced monoidal categories} in \Cref{section:trace}.
We discuss the relationship between feedback and delay in \Cref{section:delayfeedback}.
Next, we recall the construction of the \emph{free} feedback category in \Cref{section:freefeedback}, and give examples in \Cref{section:examplesfeedback}.

\subsection{Feedback Categories}\label{section:feedback}

Feedback categories~\cite{katis02} were motivated by examples such as \emph{Elgot automata}~\cite{elgot75}, \emph{iteration theories}~\cite{bloom93} and \emph{continuous dynamical systems}~\cite{katis99}. 
These categories feature a \defining{linkfeedbackoperator}{feedback operator}, $\fbk(\bullet)$, which takes a morphism $\sS \otimes \sA \to \sS \otimes \sB$ and \emph{``feeds back''} one of its outputs to one of its inputs of the same type, yielding a morphism $\sA \to \sB$ (\Cref{fig:feedbacktype}, left). 
When using string diagrams, we depict the action of the \feedbackoperator{} as a loop with a double arrowtip (\Cref{fig:graphicalnotationfeedback}, right): string diagrams must be acyclic, and so the feedback operator cannot be confused with a normal wire.
\begin{figure}[H]
  \centering
    \(\infer{\fbkOn{}{\sS}{\sA}{\sB}(\fm) \colon \sA \to \sB}{f \colon \sS \otimes \sA \to \sS \otimes \sB}\)
    \qquad\qquad
    \feedbacknotation
    \caption{Type and graphical notation for the operator $\fbkOn{}{\sS}{\sA}{\sB}(\bullet)$.}
    \label{fig:graphicalnotationfeedback}
    \label{fig:feedbacktype}
\end{figure}

Capturing a reasonable notion of feedback requires the operator to interact coherently with
the flow imposed by the structure of a symmetric monoidal category. This interaction
is expressed by a few straightforward axioms, which we list below.
\begin{definition}\label{definition:feedback}

  A \defining{linkcategorywithfeedback}{feedback category~\cite{katis02}} is a symmetric monoidal
  category \(\catC\) endowed with an operator
  \(\fbkOn{}{\sS}{A}{B} \colon \catC(\sS \otimes A, \sS \otimes B) \to \catC(A,B)\),
  which satisfies the following axioms (A1-A5, see also \Cref{figure:axioms}).

\addtolength\leftmargini{2em}
\begin{enumerate}[label={\quad(A\arabic*).}, ref={(A\arabic*)}, start=1 ]
  \item\label{axiom:tight} %
\defining{linktightening}{Tightening}. Feedback must be natural in \(\sA, \sB \in \catC\), its input and output.
This is to say that for every morphism \(\fm \colon \sS \otimes \sA \to \sS \otimes \sB\) and every pair of morphisms \(\um \colon A' \to A\) and \(\vm \colon \sB \to B'\),

\[\um \comp \fbkOn{}{\sS}{\sA}{\sB}(\fm) \comp \vm = \fbkOn{}{\sS}{\sA'}{\sB'}((\im \otimes \um) \comp \fm \comp (\im \otimes \vm)).\]

   \item\label{axiom:vanish} %
\defining{linkvanishingaxiom}{Vanishing}. Feedback on the empty tensor product, the unit, does nothing.
That is to say that, for every $\fm \colon A \to B$,
\[\fbkOn{}{\sI}{\sA}{\sB}(\fm) = \fm.\]
   \item\label{axiom:join} %
\defining{linkjoiningaxiom}{Joining}. Feedback on a monoidal pair is the same as two consecutive applications of feedback. That is to say that, for every morphism $\fm \colon \sS \tensor \sT \tensor \sA \to \sS \tensor \sT \tensor \sB$,
\[\fbkOn{}{\sT}{\sA}{\sB}(\fbkOn{}{\sS}{\sS \otimes \sA}{\sS \otimes \sB}(\fm)) = \fbkOn{}{\sS \otimes \sT}{A}{B}(f).\]
   \item\label{axiom:strength} %
\defining{linkstrengthaxiom}{Strength}. Feedback has the same result if it is taken in parallel with another morphism. That is to say that, for every morphism $\fm \colon \sS \tensor \sA \to \sS \tensor \sB$ and every morphism $\gm \colon \sA' \to \sB'$,
\[\fbkOn{}{\sS}{\sA}{\sB}(\fm) \otimes \gm = \fbkOn{}{\sS}{\sA \otimes \sA'}{\sB \otimes \sB'}(\fm \otimes \gm).\]

   \item\label{axiom:slide} %
\defining{linkslidingaxiom}{Sliding}. Feedback is invariant to applying an isomorphism ``just before'' or ``just after'' the feedback. In other words, feedback is dinatural over the isomorphisms of the category.
That is to say that for every $\fm \colon \sT \tensor \sA \to \sS \tensor \sB$ and every isomorphism $h \colon S \to T$,
\[\fbkOn{}{\sT}{\sA}{\sB}(\fm \comp (\hm \otimes \im)) =
  \fbkOn{}{\sS}{\sA}{\sB}((\hm \otimes \im) \comp \fm).\]
 \end{enumerate}
\begin{figure}
  \centering
  \begin{tabular}{cc}
    \tightone $\overset{\ref{axiom:tight}}{=}$ \tighttwo &
    \vanishingone  $\overset{\ref{axiom:vanish}}{=}$ \vanishingtwo \\
    \vanishingproductone  $\overset{\ref{axiom:join}}{=}$ \vanishingproducttwo &
   \strengthone  $\overset{\ref{axiom:strength}}{=}$ \strengthtwo
  \end{tabular}
  \begin{ceqn} %
  \begin{align*}
    \slideone  \overset{\ref{axiom:slide}}{=} \slidetwo \mbox{ ($h$ isomorphism)}
  \end{align*}
  \end{ceqn}
    \caption{Diagrammatic depiction of the axioms of feedback.}
  \label{figure:axioms}
\end{figure}

 \end{definition}

The natural notion of homomorphism between feedback categories is that of a symmetric monoidal functor that moreover preserves the feedback structure.  These are called \emph{feedback functors}.

\begin{definition}
  A \defining{linkfeedbackfunctor}{feedback functor} $F \colon \catC \to \catD$
  between two \hyperlink{linkcategorywithfeedback}{feedback categories}
  $(\catC,\fbkOn{\catC}{}{}{})$ and $(\catD,\fbkOn{\catD}{}{}{})$ is a \hyperlink{linkmonoidalfunctor}{strong symmetric monoidal functor} such that feedback is transported, that is,
  \[ F(\fbkOn{\catC}{\sS}{\sA}{\sB}(\fm)) = \fbkOn{\catD}{F(\sS)}{F(\sA)}{F(\sB)}(\mu \comp Ff \comp \mu^{-1}),\]
  where \(\mu_{A,B} \colon F(\sA) \tensor F(\sB) \to F(\sA \tensor \sB )\) is the isomorphism of the strong monoidal functor \(F\). %
  We write \defining{linkcatfeedback}{$\mathsf{Feedback}$} for the category of (small)
  \hyperlink{linkcategorywithfeedback}{feedback categories} and \hyperlink{linkfeedbackfunctor}{feedback
    functors}.  There is a forgetful functor ${\cal U} \colon \FEEDBACK \to \SYMMON$.
\end{definition}

\begin{remark}
  Thanks to the coherence theorem (\Cref{theorem:coherence}), we can present the axioms of a feedback category as in~\Cref{definition:feedback}, omitting associators and unitors.
  In fact, to be explicit, the statement of the \hyperlink{linkvanishingaxiom}{vanishing axiom} is
  \[\fbkOn{}{\sI}{}{}(\lambda_{A} \comp f \comp \inverse{\lambda_{B}}) = f\]
  because the feedback operator, $\fbkOn{}{I}{}{}$, needs to be applied to a morphism $I \otimes A \to I \otimes B$, and the only morphism whose strictification has type $A \to B$ is $(\lambda_{A} \comp f \comp \inverse{\lambda_{B}}) \colon I \otimes A \to I \otimes B$ (see~\Cref{theorem:coherence}).
  Similarly, the \hyperlink{linkjoiningaxiom}{joining axiom} really states that
  \[\fbkOn{}{\sS}{\sA}{\sB}(\fbkOn{}{\sT}{\sS \otimes \sA}{\sS \otimes \sB}(\fm)) = \fbkOn{}{\sS \otimes \sT}{A}{B}(\alpha_{\sS,\sT,A} \comp \fm \comp \alpha^{-1}_{\sS,\sT,B}).\]
\end{remark}

\begin{remark}
  Our feedback operator takes a morphism $S \otimes A \to S \otimes B$ with the first component $S$ of the tensor in both the domain and the codomain being the object ``fed back''. Given that $S$ appears in the first position in both the domain and the codomain, we refer to this as \emph{aligned feedback}.

  An alternative definition is possible, and appears in the exposition of traces by Ponto and Shulman~\cite{ponto14}.
  We call this \emph{twisted feedback}: here $\fbk(\bullet)$ is an
  operator that takes a morphism $S \otimes A \to B \otimes S$---note the position of $S$ in the codomain---and yields a
  morphism $A \to B$.
  \[\infer{\fbkOn{}{\sS}{\sA}{\sB}(\fm) \colon \sA \to \sB}{f \colon \sS \otimes \sA \to \sB \tensor \sS}\]
  The advantage of using \emph{twisted feedback} is that sequential composition of processes with feedback does not require symmetry of the underlying monoidal category (see~\cite{sabadini95}, where the authors consider a category with twisted feedback).
  However, parallel composition \emph{does} require symmetry. Given that we study the monoidal category of feedback processes, and aligned feedback diagrams are more readable, we use only aligned feedback in this paper.

  \begin{figure}[H]
    \centering
    \vspace{0.5em}
    \twistedvsaligned{}
    \caption{Twisted vs. aligned feedback}
  \end{figure}
\end{remark}

\subsection{Traced Monoidal Categories}\label{section:trace}

Feedback categories are a weakening of \tracedMonoidalCategories{}, which have found several applications in computer science.
Indeed, since their conception~\cite{joyal96} as an abstraction of the \emph{trace} of a matrix in linear algebra, they were used in linear logic and geometry of interaction~\cite{abramsky14,girard87,girard89}, programming language semantics~\cite{hasegawa97}, semantics of recursion~\cite{adamek06} and fixed point operators~\cite{hasegawa02,benton03}.

Between feedback categories and traced monoidal categories there is an intermediate notion called \emph{right traced category}~\cite{selinger10}.
Here, the sliding axiom applies not only to isomorphisms but rather to arbitrary morphisms.
This strengthening is already unsuitable for our purposes (see~\Cref{rem:sliding-isos}).
However, the difference in the sliding axiom is not dramatic: we will generalize the notion of feedback category to allow the choice of morphisms that can be ``slid'' through the feedback loop (\Cref{sec:generalizingfeedback}).
For example, it is possible to require the sliding axiom for all the morphisms, as in the case of right traced categories, or just isomorphisms, as in the case of feedback categories.
The more serious conceptual difference between feedback categories and  \tracedMonoidalCategories{} is the ``yanking axiom'' of traced monoidal categories (in \Cref{fig:theyankingaxiom}).
The \yankingAxiom{} is incontestably elegant from the geometrical point of view: strings are ``pulled'', and feedback (the loop with two arrowtips) disappears.

\begin{figure}[h]
  \centering
  \theYankingEquation
  \caption{The yanking axiom.}
  \label{fig:theyankingaxiom}
\end{figure}

Strengthening the \slidingaxiom{} and adding the \yankingAxiom{}
yields the definition of \tracedMonoidalCategory{}.
\begin{definition}
  A \defining{linktracedmonoidalcategory}{traced monoidal category} \cite{joyal96,selinger10} is a \categoryWithFeedback{} that additionally satisfies the \emph{yanking axiom} $\fbk(\sigma) = \im$ and the \emph{sliding axiom},
  \(\fbkOn{}{\sT}{\sA}{\sB}(f \comp (h \otimes \im)) = \fbkOn{}{\sS}{\sA}{\sB}((h \otimes \im) \comp f)\), for an arbitrary morphism $h \colon S \to T$.
  We commonly denote by $\mathsf{tr}(\bullet)$ the feedback operator of a traced monoidal category.
\end{definition}

\

\begin{figure}[H]
  \centering
  \norlatchtraced
  \caption{Diagram for the NOR latch, modeled with a trace in $\protect\SpanGraph$.}\label{fig:nor-latch-trace}
\end{figure}

There is scope for questioning the validity of the yanking axiom in many applications that feature feedback.
If
feedback can disappear without leaving any imprint, that must mean that it is
\emph{instantaneous}: its output necessarily mirrors its input.\footnote{In
  other words, traces are used to talk about processes in \emph{equilibrium},
  processes that have reached a \emph{fixed point}. A theorem by Hasegawa \cite{hasegawa02} and
  Hyland \cite{benton03} corroborates this interpretation: a trace in a cartesian category
  corresponds to a \emph{fixpoint operator}.} Importantly for our purposes, this
 implies that a feedback satisfying the yanking equation is
``memoryless'', or ``stateless''.

In engineering and computer science, instantaneous feedback is actually a rare concept; a more common notion is that of \emph{guarded feedback}.
Consider \emph{signal flow graphs} \cite{shannon42,mason53}: their categorical
interpretation in \cite{bonchi17} models feedback not by the usual trace, but by
a trace ``guarded by a register'', that \emph{delays the signal} and violates the yanking axiom (see Remark 7.8 \emph{op}.\emph{cit}.). %

\begin{example}
  Let us return to our running example of the NOR latch from \Cref{fig:norintro}.
  We have seen how to model NOR gates in $\SpanGraph$ in \Cref{fig:norandlatch}, and the algebra of $\SpanGraph$ does include a trace.
  However, imitating the real-world behavior of the NOR latch with \emph{just} a trace is unsatisfactory: the trace of $\SpanGraph$ is built out of stateless components, and tracing stateless components yields a stateless component (see \Cref{fig:nor-latch-trace}, later detailed in \Cref{sec:componentsspangraph}).
\end{example}

\subsection{Delay and Feedback}\label{section:delayfeedback}

As we have discussed previously, the major conceptual difference between \hyperlink{linkcategorywithfeedback}{feedback categories} and \tracedMonoidalCategories{} is the rejection of the \yankingAxiom{}.
Indeed, a non-trivial delay is what sets apart feedback categories from traced monoidal categories.

We can isolate the delay component in a feedback category.
Consider the process that only ``feeds back'' the input to itself and then just outputs that ``fed back'' input.
The process interpretation of monoidal categories (\Cref{sec:theories-processes}) allows us to understand this process as delaying its input and returning it as output~\cite{monoidalStreams22}.
This process,
$\delay_{\sA} \coloneqq \fbkOn{}{\sA}{}{}(\sigma_{\sA,\sA})$, is called the
\emph{delay endomorphism} and is illustrated in \Cref{fig:delay}.
\begin{figure}[h]
  \centering
  \definingDelay{}
  \caption{Definition of \emph{delay}.\label{fig:delay}}
\end{figure}

\medskip
If a category has enough structure, feedback can be understood as the combination of \emph{trace} and
\emph{delay} in a formal sense.
\defining{linkcompactclosed}{Compact closed categories} are \tracedMonoidalCategories{} where every object $A$
has a dual $A^{\star}$ and the trace is constructed from two pieces $\varepsilon \colon A \otimes A^{\star} \to I$ and $\eta \colon I \to A^{\star} \otimes A$.
While not every \tracedMonoidalCategory{} is compact closed, they all
embed fully faithfully into a compact closed category.\footnote{This is the $\mathbf{Int}$ construction from Joyal, Street and Verity \cite{joyal96}.}
In a \compactClosedCategory{}, a feedback operator is necessarily a trace ``guarded'' by a \emph{delay}.

\begin{figure}[h]
  \centering
  \norlatchfeedback
  \caption{NOR latch with feedback.}
  \label{fig:norwithfeedback}
\end{figure}

\begin{proposition}[Feedback from delay~\cite{bonchi19}]\label{proposition:delaycompactclosed}
  Let $\catC$ be a \compactClosedCategory{} with $\fbkOn{\catC}{}{}{}$ a
  feedback operator that takes a morphism $S \otimes A \to S \otimes B$ to a
  morphism $A \to B$, satisfying the axioms of feedback (in
  \Cref{figure:axioms}) but possibly failing to satisfy the yanking axiom (\Cref{fig:theyankingaxiom}) of
  traced monoidal categories. Then, the feedback operator is necessarily of the
  form
  \[
    \fbkOn{\catC}{\sS}{}{}(f) \coloneqq
      (\eta \otimes \im) \comp
      (\im \otimes f) \comp
      (\im \otimes \delay_{S} \otimes \im) \comp
      (\varepsilon \otimes \im)
  \]
  where $\delay_{A} \colon A \to A$ is a family of endomorphisms satisfying
  \begin{itemize}[label=$\bullet$]
    \item $\delay_{A} \otimes \delay_{B} = \delay_{A \otimes B}$ and $\delay_{I} = \im$, and
    \item $\delay_{A} \comp \hm = \hm \comp \delay_{B}$ for each isomorphism $\hm \colon A \cong B$.
  \end{itemize}
  In fact, any family of morphisms $\delay_{A}$ satisfying these properties
  determines uniquely a feedback operator that has $\delay_{A}$ as its delay
  endomorphisms.
\end{proposition}

\begin{proof}
  \begin{figure}[h!]
    \centering
    \feedbackFromDelay
    \caption{Feedback from delay.}\label{fig:feedbackfromdelay}
  \end{figure}
  Given a family $\delay_{S}$ satisfying the two
  properties, we can define a feedback structure, shown in \Cref{fig:feedbackfromdelay}, to be
  $\fbkOn{\catC}{\sS}{}{}(f) \coloneqq (\eta \otimes \im) \comp (\im \otimes f) \comp (\im \otimes \delay_{S} \otimes \im) \comp (\varepsilon \otimes \im)$
  and check that it satisfies all the axioms of feedback (\Cref{figure:axioms}).
  Note here that, as expected, the yanking equation is satisfied precisely when
  delay endomorphisms are identities, $\delay_{A} = \im_{A}$.

  Let us now show that any feedback operator in a compact closed category is
  of this form (\Cref{fig:compactproof}). Indeed,
  \begin{align*}
    \fbkOn{\catC}{\sS}{}{}(f) &=
    \fbkOn{\catC}{\sS}{}{}(
       (\im \otimes \eta \otimes \eta \otimes \im) \comp
       (\sigma \otimes \sigma \otimes f) \comp
       (\im \otimes \varepsilon \otimes \varepsilon \otimes \im)
       ) \\ &=
       (\im \otimes \eta \otimes \eta \otimes \im) \comp
       (\fbkOn{\catC}{\sS}{}{}(\sigma) \otimes \sigma \otimes f) \comp
       (\im \otimes \varepsilon \otimes \varepsilon \otimes \im) \\ &=
       (\eta \otimes \im) \comp (\im \otimes f) \comp (\im \otimes \fbkOn{\catC}{\sS}{}{}(\sigma) \otimes \im) \comp (\varepsilon \otimes \im).
  \end{align*}

  \begin{figure}
    \centering
\tikzset{every picture/.style={line width=0.75pt}} %
\begin{tikzpicture}[x=0.75pt,y=0.75pt,yscale=-1,xscale=1]
\draw    (50,160) ;
\draw [shift={(50,160)}, rotate = 0] [color={rgb, 255:red, 0; green, 0; blue, 0 }  ][line width=0.75]    (17.64,-3.29) .. controls (13.66,-1.4) and (10.02,-0.3) .. (6.71,0) .. controls (10.02,0.3) and (13.66,1.4) .. (17.64,3.29)(10.93,-3.29) .. controls (6.95,-1.4) and (3.31,-0.3) .. (0,0) .. controls (3.31,0.3) and (6.95,1.4) .. (10.93,3.29)   ;
\draw   (50,170) -- (70,170) -- (70,210) -- (50,210) -- cycle ;
\draw    (30,200) -- (50,200) ;
\draw    (50,160) .. controls (30,160.25) and (30.5,179.25) .. (50,180) ;
\draw    (50,160) -- (70,160) ;
\draw    (70,160) .. controls (90.4,160.6) and (90,179.75) .. (70,180) ;
\draw    (70,200) -- (90,200) ;
\draw    (230,130.01) ;
\draw [shift={(230,130.01)}, rotate = 0] [color={rgb, 255:red, 0; green, 0; blue, 0 }  ][line width=0.75]    (17.64,-3.29) .. controls (13.66,-1.4) and (10.02,-0.3) .. (6.71,0) .. controls (10.02,0.3) and (13.66,1.4) .. (17.64,3.29)(10.93,-3.29) .. controls (6.95,-1.4) and (3.31,-0.3) .. (0,0) .. controls (3.31,0.3) and (6.95,1.4) .. (10.93,3.29)   ;
\draw   (230,170.01) -- (250,170.01) -- (250,210.01) -- (230,210.01) -- cycle ;
\draw    (210,200.01) -- (230,200.01) ;
\draw    (226.23,130) .. controls (206.23,130.25) and (207.03,140) .. (216.23,140) ;
\draw    (226.23,130) -- (246.23,130) ;
\draw    (246.23,130) .. controls (266.63,130.6) and (266.63,140.4) .. (256.23,140) ;
\draw    (250,200.01) -- (270,200.01) ;
\draw    (216.23,140) .. controls (230.8,140) and (236,150.8) .. (246.23,150) ;
\draw    (226.23,150) .. controls (238.45,150.48) and (246,140.8) .. (256.23,140) ;
\draw    (226.23,150) .. controls (206.23,150.25) and (207.03,160) .. (216.23,160) ;
\draw    (246.23,150) .. controls (266.63,150.6) and (266.63,160.4) .. (256.23,160) ;
\draw    (216.23,160) .. controls (230.8,160) and (249.77,170.81) .. (260,170.01) ;
\draw    (220,170.01) .. controls (227.77,170.02) and (237.2,160.81) .. (256.23,160) ;
\draw    (220,170.01) .. controls (210.8,170.81) and (202.8,181.21) .. (230,180.01) ;
\draw    (260,170.01) .. controls (269.6,170.41) and (271.2,180.81) .. (250,180.01) ;
\draw    (140,149.98) ;
\draw [shift={(140,149.98)}, rotate = 0] [color={rgb, 255:red, 0; green, 0; blue, 0 }  ][line width=0.75]    (17.64,-3.29) .. controls (13.66,-1.4) and (10.02,-0.3) .. (6.71,0) .. controls (10.02,0.3) and (13.66,1.4) .. (17.64,3.29)(10.93,-3.29) .. controls (6.95,-1.4) and (3.31,-0.3) .. (0,0) .. controls (3.31,0.3) and (6.95,1.4) .. (10.93,3.29)   ;
\draw   (140,170) -- (160,170) -- (160,210) -- (140,210) -- cycle ;
\draw    (120,200) -- (140,200) ;
\draw    (136.23,149.97) .. controls (116.23,150.22) and (117.03,159.97) .. (126.23,159.97) ;
\draw    (136.23,149.98) -- (160,149.99) ;
\draw    (160,149.99) .. controls (180.4,150.59) and (180.4,160.39) .. (170,159.99) ;
\draw    (160,200) -- (180,200) ;
\draw    (126.23,159.97) .. controls (137.03,159.57) and (139.6,170.39) .. (130,169.99) ;
\draw    (130,170) .. controls (120.8,170.8) and (112.8,181.2) .. (140,180) ;
\draw    (173.77,170) .. controls (182,169.99) and (184.4,180.79) .. (160,179.99) ;
\draw    (170,159.99) .. controls (163.77,160) and (163.77,170) .. (173.77,170) ;
\draw   (320,170.01) -- (340,170.01) -- (340,210.01) -- (320,210.01) -- cycle ;
\draw    (300,200.01) -- (320,200.01) ;
\draw    (340,200.01) -- (360,200.01) ;
\draw    (320,150) .. controls (300,150.25) and (297.03,160) .. (306.23,160) ;
\draw    (340,150) .. controls (360.4,150.6) and (356.63,160.4) .. (346.23,160) ;
\draw    (306.23,160) .. controls (320.8,160) and (339.77,170.81) .. (350,170.01) ;
\draw    (310,170.01) .. controls (317.77,170.02) and (327.2,160.81) .. (346.23,160) ;
\draw    (310,170.01) .. controls (300.8,170.81) and (292.8,181.21) .. (320,180.01) ;
\draw    (350,170.01) .. controls (359.6,170.41) and (361.2,180.81) .. (340,180.01) ;
\draw   (320,140) -- (330,140) .. controls (335.52,140) and (340,144.48) .. (340,150) .. controls (340,155.52) and (335.52,160) .. (330,160) -- (320,160) -- cycle ;
\draw   (410,170) -- (430,170) -- (430,210) -- (410,210) -- cycle ;
\draw    (430,200) -- (480,200) ;
\draw    (390,200) -- (410,200) ;
\draw    (460,160) .. controls (480.5,159.75) and (480,179.75) .. (460,180) ;
\draw    (410,160) .. controls (390,160.25) and (390.5,179.25) .. (410,180) ;
\draw    (410,160) -- (460,160) ;
\draw    (430,180) -- (440,180) ;
\draw   (440,170) -- (450,170) .. controls (455.52,170) and (460,174.47) .. (460,180) .. controls (460,185.52) and (455.52,190) .. (450,190) -- (440,190) -- cycle ;
\draw (60,190) node  [font=\normalsize]  {$f$};
\draw (37,207) node  [font=\scriptsize]  {$A$};
\draw (83.5,206.5) node  [font=\scriptsize]  {$B$};
\draw (63.5,147.5) node  [font=\footnotesize]  {$S$};
\draw (216.5,206.5) node  [font=\scriptsize]  {$A$};
\draw (173.5,206.5) node  [font=\scriptsize]  {$B$};
\draw (236.5,117.5) node  [font=\footnotesize]  {$S$};
\draw (126.5,206.5) node  [font=\scriptsize]  {$A$};
\draw (263.5,206.5) node  [font=\scriptsize]  {$B$};
\draw (150,190) node  [font=\normalsize]  {$f$};
\draw (146.5,137.5) node  [font=\footnotesize]  {$S$};
\draw (240,190.01) node  [font=\normalsize]  {$f$};
\draw (91,182.4) node [anchor=north west][inner sep=0.75pt]    {$ =$};
\draw (181,182.4) node [anchor=north west][inner sep=0.75pt]    {$ =$};
\draw (306.5,206.5) node  [font=\scriptsize]  {$A$};
\draw (353.5,206.5) node  [font=\scriptsize]  {$B$};
\draw (330,190.01) node  [font=\normalsize]  {$f$};
\draw (330,150) node  [font=\footnotesize]  {$\partial$};
\draw (271,182.4) node [anchor=north west][inner sep=0.75pt]    {$=$};
\draw (420,190) node  [font=\normalsize]  {$f$};
\draw (396.5,206.5) node  [font=\scriptsize]  {$A$};
\draw (473.5,206.5) node  [font=\scriptsize]  {$B$};
\draw (437,152) node  [font=\footnotesize]  {$S$};
\draw (450,180) node  [font=\footnotesize]  {$\partial$};
\draw (361,183.4) node [anchor=north west][inner sep=0.75pt]    {$=$};
\end{tikzpicture}
\caption{Feedback in a compact closed category.}\label{fig:compactproof}
\end{figure}

  Here we have used the fact that the trace is constructed by two separate pieces:
  $\varepsilon$ and $\eta$; and then the fact that the feedback operator, like trace, can be applied ``locally'' (see the axioms in \Cref{figure:axioms}).
\end{proof}

\begin{example}
  Consider again the NOR latch of \Cref{fig:norintro}.
  The algebra of the category $\SpanGraph$ does include a feedback operator that is \emph{not} a trace -- the difference is an additional \emph{stateful} delay component.
  As we shall see, this notion of feedback is canonical.
We shall also see that the delay enables us to capture
the real-world behavior of the NOR latch (\Cref{fig:norwithfeedback}).
\end{example}
The emergence of state from feedback is witnessed by the $\St(\bullet)$ construction, which we recall below.

\subsection{$\mathsf{St}(\bullet)$, the Free Feedback Category}\label{section:freefeedback}

Here we show how to obtain the free feedback category on a symmetric monoidal category. The $\St(\bullet)$ construction is
a general way of endowing a system with state. It appears multiple times in
the literature in slightly different forms: it is used to arrive at a stateful resource
calculus in~\cite{bonchi19}; a variant is used for geometry of interaction in~\cite{hoshino14}; it coincides with the free \categoryWithFeedback{} presented in~\cite{katis02}; and yet another, slightly different formulation was given in~\cite{sabadini95}.

\begin{definition}[Category of stateful processes,~\cite{katis02}]\label{def:sequentialcomposition} \nicelinktarget{linkFbk} \nicelinktarget{linkfeedbackcircuit}
  Let $(\catC, \otimes, I)$ be a \hyperlink{linksymmetricmonoidalcategory}{symmetric monoidal category}. We write
  ${\color{NavyBlue}{\mathsf{St}}}(\catC)$ for the category with the objects of $\catC$ but where
  morphisms $A \to B$ are pairs $(S \mid f)$, consisting of a \emph{state space}
  $S \in \catC$ and a morphism $f \colon S \otimes A \to S \otimes B$. We consider
  morphisms up to isomorphism classes of their state space, and thus
  \[(\sS \mid \fm) = (\sT \mid (h^{-1} \otimes \im) \comp \fm \comp (h \otimes \im)), \quad \mbox{ for any isomorphism } h \colon \sS \cong \sT.\]
  When depicting a \statefulProcess{} (\Cref{fig:statefulprocessdepiction}), we mark the state strings.
  \begin{figure}[H]
   \centering
    \cdDefA = \cdDefAB
    \caption{Equivalence of stateful processes. We depict \protect\statefulProcesses{} by marking the space state.}
    \label{fig:statefulprocessdepiction}
  \end{figure}
\end{definition}

We define the \defining{linkidentitycircuit}{identity \statefulProcessName{}} on \(\sA \in \catC\) as \((\sI \mid \im_{I \otimes A})\).
\defining{linksequentialcomposition}{Sequential composition} of the two \statefulProcesses{} \((\sS \mid \fm) \colon \sA \to \sB\) and \((\sT \mid \gm) \colon \sB \to \sC\) is defined by
\((\sS \mid \fm) \comp (\sT \mid \gm) =
  (\sS \otimes \sT \mid (\sigma \otimes \im) \comp
  (\im \otimes \fm) \comp
  (\sigma \otimes \im) \comp
  (\im \otimes \gm))\), see \Cref{fig:sequential-and-parallel}, left.
\label{def:parallelcomposition}
\defining{linkparallelcomposition}{Parallel composition} of the two \statefulProcesses{} \((S \mid f) \colon A \to B\) and \((S' \mid f') \colon A' \to B'\) is defined by
  \((S \mid f) \otimes (S' \mid f') =
   (S \tensor S' \mid (\im \tensor \sigma \tensor \im) \comp
   (f \tensor f') \comp
   (\im \tensor \sigma \tensor \im))\), see \Cref{fig:sequential-and-parallel}, right.
   In both cases, the state spaces of the components are tensored together.
\begin{figure}[H]
   \centering
    \sequentialCompositionOfCircuits \quad \cdParCompA
    \caption{\protect\hyperlink{linksequentialcomposition}{Sequential} and
      \protect\hyperlink{linkparallelcomposition}{parallel composition} of
    \protect\statefulProcesses.}\label{fig:sequential-and-parallel}
\end{figure}

This defines a symmetric monoidal category. Moreover, the operator
\[\defining{linkstore}{\ensuremath{\store{\sT}}}(\sS \mid \fm) \coloneqq (\sS \otimes \sT \mid \fm)\mbox{, for } f \colon \sS \tensor \sT \tensor \sA \to \sS \tensor \sT \tensor \sB,\]
which ``stores'' some information into the state, makes it a feedback category, see \Cref{fig:storediagram}.
\begin{figure}[H]
  \centering
  $\store{T} \left(\cdPreFeedback \right)\ =\ \cdPreFeedbackB$
\caption{The $\protect\store{}(\bullet)$ operation, diagrammatically.}\label{fig:storediagram}
\end{figure}

\begin{proposition}
  Sequential composition of stateful processes is associative. That is, for every
  $\fm \colon \sS \tensor \sA \to \sS \tensor \sB$, every $\gm \colon \sT \tensor \sB \to \sT \tensor \sC$ and every $\hm \colon \sR \tensor \sC \to \sR \tensor \sD$,
  \[((\sS \mid \fm) \comp (\sT \mid \gm)) \comp (\sR \mid \hm) =
  (\sS \mid \fm) \comp ((\sT \mid \gm) \comp (\sR \mid \hm)).\]
\end{proposition}
\begin{proof}
  We can see both morphisms are equal by applying transformations of string diagrams: i.e. the axioms of symmetric monoidal categories (\Cref{fig:associativityproof}).
  \begin{figure}[!h]
  \[\begin{tikzpicture}[x=0.75pt,y=0.75pt,yscale=-1,xscale=1]
\draw   (110,60) -- (130,60) -- (130,100) -- (110,100) -- cycle ;
\draw    (130,90) -- (180,90) ;
\draw    (20,90) -- (110,90) ;
\draw    (70,50) .. controls (85,50) and (85.33,69.67) .. (100,70) ;
\draw    (70,70) .. controls (85,69.33) and (85.67,50.33) .. (100,50) ;
\draw    (100,50) -- (140,50) ;
\draw    (140,50) .. controls (155,50) and (155.33,69.67) .. (170,70) ;
\draw    (140,70) .. controls (155,69.33) and (155.67,50.33) .. (170,50) ;
\draw    (100,70) -- (110,70) ;
\draw    (130,70) -- (140,70) ;
\draw   (180,60) -- (200,60) -- (200,100) -- (180,100) -- cycle ;
\draw    (200,90) -- (260,90) ;
\draw    (170,50) -- (220,50) ;
\draw    (170,70) -- (180,70) ;
\draw    (200,70) -- (220,70) ;
\draw    (30,70) .. controls (45,69.33) and (45.67,30.33) .. (60,30) ;
\draw    (30,30) .. controls (45,30) and (45.33,49.67) .. (60,50) ;
\draw    (30,50) .. controls (45,50) and (45.33,69.67) .. (60,70) ;
\draw    (60,70) -- (70,70) ;
\draw    (60,50) -- (70,50) ;
\draw    (20,30) -- (30,30) ;
\draw    (20,50) -- (30,50) ;
\draw    (20,70) -- (30,70) ;
\draw    (60,30) -- (220,30) ;
\draw    (220,30) .. controls (236.2,30.6) and (235.67,70.33) .. (250,70) ;
\draw    (220,70) .. controls (235,70) and (235.33,49.67) .. (250,50) ;
\draw    (220,50) .. controls (235,50) and (235.33,29.67) .. (250,30) ;
\draw    (250,30) -- (300,30) ;
\draw    (250,50) -- (300,50) ;
\draw    (250,70) -- (260,70) ;
\draw   (260,60) -- (280,60) -- (280,100) -- (260,100) -- cycle ;
\draw    (280,70) -- (300,70) ;
\draw    (280,90) -- (300,90) ;
\draw  [color={rgb, 255:red, 155; green, 155; blue, 155 }  ,draw opacity=1 ][fill={rgb, 255:red, 155; green, 155; blue, 155 }  ,fill opacity=1 ] (17.46,30) .. controls (17.46,28.62) and (18.6,27.5) .. (20,27.5) .. controls (21.4,27.5) and (22.54,28.62) .. (22.54,30) .. controls (22.54,31.38) and (21.4,32.5) .. (20,32.5) .. controls (18.6,32.5) and (17.46,31.38) .. (17.46,30) -- cycle ;
\draw  [color={rgb, 255:red, 155; green, 155; blue, 155 }  ,draw opacity=1 ][fill={rgb, 255:red, 155; green, 155; blue, 155 }  ,fill opacity=1 ] (17.46,50) .. controls (17.46,48.62) and (18.6,47.5) .. (20,47.5) .. controls (21.4,47.5) and (22.54,48.62) .. (22.54,50) .. controls (22.54,51.38) and (21.4,52.5) .. (20,52.5) .. controls (18.6,52.5) and (17.46,51.38) .. (17.46,50) -- cycle ;
\draw  [color={rgb, 255:red, 155; green, 155; blue, 155 }  ,draw opacity=1 ][fill={rgb, 255:red, 155; green, 155; blue, 155 }  ,fill opacity=1 ] (17.46,70) .. controls (17.46,68.62) and (18.6,67.5) .. (20,67.5) .. controls (21.4,67.5) and (22.54,68.62) .. (22.54,70) .. controls (22.54,71.38) and (21.4,72.5) .. (20,72.5) .. controls (18.6,72.5) and (17.46,71.38) .. (17.46,70) -- cycle ;
\draw  [color={rgb, 255:red, 155; green, 155; blue, 155 }  ,draw opacity=1 ][fill={rgb, 255:red, 155; green, 155; blue, 155 }  ,fill opacity=1 ] (297.46,30) .. controls (297.46,28.62) and (298.6,27.5) .. (300,27.5) .. controls (301.4,27.5) and (302.54,28.62) .. (302.54,30) .. controls (302.54,31.38) and (301.4,32.5) .. (300,32.5) .. controls (298.6,32.5) and (297.46,31.38) .. (297.46,30) -- cycle ;
\draw  [color={rgb, 255:red, 155; green, 155; blue, 155 }  ,draw opacity=1 ][fill={rgb, 255:red, 155; green, 155; blue, 155 }  ,fill opacity=1 ] (297.46,50) .. controls (297.46,48.62) and (298.6,47.5) .. (300,47.5) .. controls (301.4,47.5) and (302.54,48.62) .. (302.54,50) .. controls (302.54,51.38) and (301.4,52.5) .. (300,52.5) .. controls (298.6,52.5) and (297.46,51.38) .. (297.46,50) -- cycle ;
\draw  [color={rgb, 255:red, 155; green, 155; blue, 155 }  ,draw opacity=1 ][fill={rgb, 255:red, 155; green, 155; blue, 155 }  ,fill opacity=1 ] (297.46,70) .. controls (297.46,68.62) and (298.6,67.5) .. (300,67.5) .. controls (301.4,67.5) and (302.54,68.62) .. (302.54,70) .. controls (302.54,71.38) and (301.4,72.5) .. (300,72.5) .. controls (298.6,72.5) and (297.46,71.38) .. (297.46,70) -- cycle ;
\draw  [color={rgb, 255:red, 155; green, 155; blue, 155 }  ,draw opacity=1 ][dash pattern={on 4.5pt off 4.5pt}] (70,40) -- (210,40) -- (210,110) -- (70,110) -- cycle ;
\draw (120,80) node    {$f$};
\draw (33,97) node  [font=\scriptsize]  {$A$};
\draw (147,97) node  [font=\scriptsize]  {$B$};
\draw (190,80) node    {$g$};
\draw (217,97) node  [font=\scriptsize]  {$C$};
\draw (27,77) node  [font=\scriptsize]  {$R$};
\draw (270,80) node    {$h$};
\draw (293,97) node  [font=\scriptsize]  {$D$};
\draw (27,57) node  [font=\scriptsize]  {$T$};
\draw (27,37) node  [font=\scriptsize]  {$S$};
\end{tikzpicture}\]
\[\begin{tikzpicture}[x=0.75pt,y=0.75pt,yscale=-1,xscale=1]
\draw   (70,60) -- (90,60) -- (90,100) -- (70,100) -- cycle ;
\draw    (100,90) -- (180,90) ;
\draw    (20,90) -- (70,90) ;
\draw    (100,50) .. controls (115,50) and (115.33,69.67) .. (130,70) ;
\draw    (140,70) .. controls (155,69.33) and (155.67,50.33) .. (170,50) ;
\draw    (90,70) -- (100,70) ;
\draw   (180,60) -- (200,60) -- (200,100) -- (180,100) -- cycle ;
\draw    (200,90) -- (250,90) ;
\draw    (170,50) -- (210,50) ;
\draw    (170,70) -- (180,70) ;
\draw    (200,70) -- (210,70) ;
\draw    (30,70) .. controls (45,69.33) and (45.67,50.33) .. (60,50) ;
\draw    (30,30) .. controls (45,30) and (45.33,69.67) .. (60,70) ;
\draw    (60,70) -- (70,70) ;
\draw    (60,50) -- (100,50) ;
\draw    (20,30) -- (30,30) ;
\draw    (20,50) -- (30,50) ;
\draw    (20,70) -- (30,70) ;
\draw    (60,30) -- (100,30) ;
\draw    (210,50) .. controls (226.2,50.6) and (225.67,70.34) .. (240,70) ;
\draw    (210,70) .. controls (225,70) and (225.33,49.67) .. (240,50) ;
\draw    (240,30) -- (290,30) ;
\draw    (240,50) -- (290,50) ;
\draw    (240,70) -- (250,70) ;
\draw   (250,60) -- (270,60) -- (270,100) -- (250,100) -- cycle ;
\draw    (270,70) -- (290,70) ;
\draw    (270,90) -- (290,90) ;
\draw  [color={rgb, 255:red, 155; green, 155; blue, 155 }  ,draw opacity=1 ][fill={rgb, 255:red, 155; green, 155; blue, 155 }  ,fill opacity=1 ] (17.46,30) .. controls (17.46,28.62) and (18.6,27.5) .. (20,27.5) .. controls (21.4,27.5) and (22.54,28.62) .. (22.54,30) .. controls (22.54,31.38) and (21.4,32.5) .. (20,32.5) .. controls (18.6,32.5) and (17.46,31.38) .. (17.46,30) -- cycle ;
\draw  [color={rgb, 255:red, 155; green, 155; blue, 155 }  ,draw opacity=1 ][fill={rgb, 255:red, 155; green, 155; blue, 155 }  ,fill opacity=1 ] (17.46,50) .. controls (17.46,48.62) and (18.6,47.5) .. (20,47.5) .. controls (21.4,47.5) and (22.54,48.62) .. (22.54,50) .. controls (22.54,51.38) and (21.4,52.5) .. (20,52.5) .. controls (18.6,52.5) and (17.46,51.38) .. (17.46,50) -- cycle ;
\draw  [color={rgb, 255:red, 155; green, 155; blue, 155 }  ,draw opacity=1 ][fill={rgb, 255:red, 155; green, 155; blue, 155 }  ,fill opacity=1 ] (17.46,70) .. controls (17.46,68.62) and (18.6,67.5) .. (20,67.5) .. controls (21.4,67.5) and (22.54,68.62) .. (22.54,70) .. controls (22.54,71.38) and (21.4,72.5) .. (20,72.5) .. controls (18.6,72.5) and (17.46,71.38) .. (17.46,70) -- cycle ;
\draw  [color={rgb, 255:red, 155; green, 155; blue, 155 }  ,draw opacity=1 ][fill={rgb, 255:red, 155; green, 155; blue, 155 }  ,fill opacity=1 ] (287.46,30) .. controls (287.46,28.62) and (288.6,27.5) .. (290,27.5) .. controls (291.4,27.5) and (292.54,28.62) .. (292.54,30) .. controls (292.54,31.38) and (291.4,32.5) .. (290,32.5) .. controls (288.6,32.5) and (287.46,31.38) .. (287.46,30) -- cycle ;
\draw  [color={rgb, 255:red, 155; green, 155; blue, 155 }  ,draw opacity=1 ][fill={rgb, 255:red, 155; green, 155; blue, 155 }  ,fill opacity=1 ] (287.46,50) .. controls (287.46,48.62) and (288.6,47.5) .. (290,47.5) .. controls (291.4,47.5) and (292.54,48.62) .. (292.54,50) .. controls (292.54,51.38) and (291.4,52.5) .. (290,52.5) .. controls (288.6,52.5) and (287.46,51.38) .. (287.46,50) -- cycle ;
\draw  [color={rgb, 255:red, 155; green, 155; blue, 155 }  ,draw opacity=1 ][fill={rgb, 255:red, 155; green, 155; blue, 155 }  ,fill opacity=1 ] (287.46,70) .. controls (287.46,68.62) and (288.6,67.5) .. (290,67.5) .. controls (291.4,67.5) and (292.54,68.62) .. (292.54,70) .. controls (292.54,71.38) and (291.4,72.5) .. (290,72.5) .. controls (288.6,72.5) and (287.46,71.38) .. (287.46,70) -- cycle ;
\draw  [color={rgb, 255:red, 155; green, 155; blue, 155 }  ,draw opacity=1 ][dash pattern={on 4.5pt off 4.5pt}] (140,40) -- (280,40) -- (280,110) -- (140,110) -- cycle ;
\draw    (30,50) .. controls (45,49.33) and (45.67,30.33) .. (60,30) ;
\draw    (90,90) -- (100,90) ;
\draw    (100,70) .. controls (115,70) and (115.33,29.67) .. (130,30) ;
\draw    (100,30) .. controls (115,30) and (115.33,49.67) .. (130,50) ;
\draw    (140,50) .. controls (155,50) and (155.33,69.67) .. (170,70) ;
\draw    (130,50) -- (140,50) ;
\draw    (130,70) -- (140,70) ;
\draw    (130,30) -- (240,30) ;
\draw (80,80) node    {$f$};
\draw (33,97) node  [font=\scriptsize]  {$A$};
\draw (147,97) node  [font=\scriptsize]  {$B$};
\draw (190,80) node    {$g$};
\draw (213,97) node  [font=\scriptsize]  {$C$};
\draw (27,77) node  [font=\scriptsize]  {$R$};
\draw (260,80) node    {$h$};
\draw (287,97) node  [font=\scriptsize]  {$D$};
\draw (27,57) node  [font=\scriptsize]  {$T$};
\draw (27,37) node  [font=\scriptsize]  {$S$};
\end{tikzpicture}\]
\caption{Associativity of sequential composition.}
\label{fig:associativityproof}
\end{figure}

The state spaces are isomorphic thanks to the associator $\alpha \colon (\sS \tensor \sT) \tensor \sR \to \sS \tensor (\sT \tensor \sR)$.
\end{proof}

Unitality and monoidality of stateful processes follow a similar reasoning.
These properties yield the following result.

\begin{theorem}[\cite{katis02}, Proposition 2.6]\label{th:storeisfree}
  The category \(\St(\catC)\), endowed with the \(\store{}(\bullet)\) operator, is the free feedback category over a \hyperlink{linksymmetricmonoidalcategory}{symmetric monoidal category} \(\catC\).
\end{theorem}
\begin{remark}\label{rem:sliding-isos}
  Stateful processes are defined \emph{up to isomorphism of the state space}.
  This is captured by axiom~\ref{axiom:slide} of feedback categories and, as mentioned in~\Cref{section:trace}, relaxing it to allow sliding of arbitrary morphisms, would yield a notion of equality of stateful processes that would be too strong for our purposes: it would equate automata with a different number of states and boundary behavior (\Cref{ex:non-equivalent-automata}).
  Considering stronger notions of equivalence of processes is possible and leads to interesting models of computation~\cite{monoidalStreams22}.
  Expanding this line of research is outside the scope of the present manuscript.
\end{remark}

\begin{remark}[Coherence and sliding]
There are cases where we do need to be careful about the correct use of associators and unitors.
For instance, we could be tempted to conclude that coherence implies that, for any $f \colon ((S \tensor T) \tensor R) \tensor A \to ((S \tensor T) \tensor R) \tensor B$, the following equation holds
$((S \tensor T) \tensor R \mid f) = (S \tensor (T \tensor R) \mid f)$ without needing to invoke the equivalence relation of stateful processes.
This would allow us to construct the category $\St(\bullet)$ of \statefulProcesses{} without having to quotient them by the equivalence relation.
However, this equality is only enabled by the fact that $\alpha_{S,T,R}$ is an isomorphism: we have
\[((S \tensor T) \tensor R \mid f) =
  (S \tensor (T \tensor R) \mid \alpha_{S,R,T} \comp f \comp \alpha_{S,R,T}^{\tiny -\!1}),\]
even if we write the equation omitting the coherence maps.  This is also what will allow us to notate \statefulProcesses{}
diagramatically. We will mark the wires forming the state space; the order in which they are tensored
does not matter thanks again to the equivalence relation that we are imposing.
\end{remark}

\subsection{Examples}\label{section:examplesfeedback}

All \emph{traced monoidal categories} are feedback categories, since the axioms of feedback are a strict weakening of the axioms of trace.
A more interesting source of examples is the $\St(\bullet)$ construction we just defined.
We present some examples of state constructions below.

\begin{example}[Mealy transition systems]\label{ex:feedback-mealy}
  A \defining{linkmealytransitionsystem}{Mealy deterministic transition system} with boundaries \(\sA\) and \(\sB\), and state space \(S\) was defined~\cite[\S 2.1]{mealy1955automata} to be just a function \(\fm \colon \sS \times \sA \to \sS \times \sB\).
  It is not difficult to see that, up to isomorphism of the state space, they are morphisms of $\St(\Set)$.
  They compose following \Cref{def:sequentialcomposition}, and form a feedback category $\mathbf{Mealy} \coloneqq \St(\Set)$.

  \begin{definition}
    A \emph{Mealy transition system} from \(\sA\) to \(\sB\) is a tuple $\mathbf{M} = (\sS, t, o)$, where $\sS$ is a set called the \emph{state space}, $t \colon \sS \times \sA \to \sS$ is a function called the \emph{transition function}, and $o \colon \sS \times \sA \to \sB$ is a function called the \emph{output function}.

    Two Mealy transition systems are equal whenever their transition functions are equal up to isomorphism of the state space. That is, two deterministic transition
    systems $\mathbf{M} = (\sS, t_M, o_M)$ and $\mathbf{N} = (\sT, t_N, o_N)$
    are considered equal whenever there exists an isomorphism \(h \colon S \cong T\) between their state spaces such that
    \[h(t_M(s,a)) = t_N(h(s),a) \quad \mbox{and} \quad o_M(s,a) = o_N(s,a).\]
    
    Whenever $t(s_0,a) = s_1$ and $o(s_0,a) = b$, we write $s_0 \overset{a/b}\rightarrow s_1$.
    We may also write a transition and output in a single function, $f(s_0,a) = (t(s_0,a),o(s_0,a)) = (s_1,b)$.
  \end{definition}

  The feedback of $\mathbf{Mealy}$ transition systems transforms input/output pairs into states.
  \Cref{fig:mealys} is an example: a transition system with a single state becomes a transition system with two states, $\{s_1,s_0\}$. We compute this feedback by transforming each transition  $(s_{i},i/s_{o})$ into a transition $(i/)$ from $s_{i}$ to $s_{o}$.
\end{example}

\begin{figure}[H]
    \centering
    $\fbkOn{}{}{}{}\left(\mealyA\right) = \left(\mealyB\right)$
    \caption{Feedback of a Mealy transition system. Every transition has a label $i/o$ indicating inputs ($i$) and outputs ($o$).}
    \label{fig:mealys}
  \end{figure}

\begin{example}[Elgot automata]
  Similarly, when we consider $\Set$ with the monoidal structure given by the disjoint union, we recover \emph{Elgot automata}~\cite{elgot75}, which are given by a transition function $S + A \to S + B$.
  These transition systems motivate the work of Katis, Sabadini and Walters in~\cite{sabadini95,katis02}.

  \begin{definition}
    An \emph{Elgot transition system} with initial states in \(\sA\) and final states in \(\sB\) is a tuple $\mathbf{E} = (\sS, p, d)$ where $\sS$ is a set called the \emph{state space}, $p \colon \sA \to \sS + \sB$ is a function called \emph{initial step} and $d \colon \sS \to \sS + \sB$ is a function called \emph{iterative step}.

    An Elgot transition system is interpreted as follows. We start by providing an initial state $\sA$. We then compute the initial step $p(a)$ which can result either in an internal state $p(a) = s \in S$ or in a final state $p(a) = b \in B$. In the later case, we are done and we return $b \in B$; in the former case, we repeatedly apply the iterative step: $d(p(a)), d(d(p(a))), \dots$ until we reach a final state.
  \end{definition}
\end{example}

\begin{example}[Linear dynamical systems]\label{ex:feedback-dynamical}
  A \emph{linear dynamical system} with inputs in \(\reals^{n}\), outputs in \(\reals^{m}\) and state space \(\reals^{k}\) is given by a number $k$, representing the dimension of the state space, and a matrix over the real numbers~\cite{kalman1969systemtheory}
  \[\left( \begin{matrix} A & B\\ C & D \end{matrix}\right)  \in \Mat(k + m, k + n).\]
  Two linear dynamical systems,
  \[\left( \begin{matrix} A & B\\ C & D \end{matrix}\right)\mbox{ and }\left( \begin{matrix} A' & B'\\ C' & D \end{matrix}\right),\]
  are considered equivalent if there is an invertible matrix \(H \in \Mat(k,k)\) such that
  \(A' = \inverse{H} A H\), \(B' = B H\), and \(C' = \inverse{H} C\).

  Linear dynamical systems are morphisms of a feedback category which coincides with $\St(\Mat)$, the free feedback category over the category of matrices \(\Mat\) as defined in \Cref{definition:matcategory}.
  The feedback operator is defined by
  \[\fbkIsi{l}\left(k, \left(\begin{matrix} A_{1} & A_{2} & B_{1} \\ A_{3} & A_{4} & B_{2} \\ C_{1} & C_{2} & D \end{matrix}\right)\right) = \left(k+l, \left(\begin{matrix} A_{1} & A_{2} & B_{1} \\ A_{3} & A_{4} & B_{2} \\ C_{1} & C_{2} & D \end{matrix}\right)\right),\]
   where  $\left(\begin{smallmatrix} A_{1} & A_{2} & B_{1} \\ A_{3} & A_{4} & B_{2} \\ C_{1} & C_{2} & D \end{smallmatrix}\right) \in \Mat(k+l+m, k+l+n)$.
\end{example}

 \newpage

\section{Span(Graph): an Algebra of Transition Systems}
\label{section:spangraph}

$\SpanGraph$~\cite{katis97} is an algebra of ``open transition systems''. It has
applications in \emph{concurrency theory} and \emph{verification}
\cite{sabadini95,katis97,katis00,sabadini18,gianola17}, and has
been recently applied to biological systems \cite{gianola20a,gianola20b}. Just
as ordinary Petri nets have an underlying (firing) semantics in terms of
transition systems, $\SpanGraph$ is used as a semantic universe for a variant of
open Petri nets, see~\cite{Sobocinski2010,Bruni2011}.

An \emph{open transition system} is a morphism of $\SpanGraph$: a transition graph endowed with two \emph{boundaries} or \emph{communication ports}.
Each transition has an effect on each boundary, and this data is used for synchronization.
This conceptual picture actually describes a subcategory, $\SpanGraphV$, where boundaries are mere sets: the alphabets of synchronization signals.
We shall recall the details of $\SpanGraphV$ and prove that it is universal, our main result:
\begin{center}
  \(\SpanGraphV\) is the free \hyperlink{linkcategorywithfeedback}{feedback category} over
  \(\Span(\Set)\).
\end{center}

\subsection{The Algebra of Spans}\label{sec:algebraofspans}

\begin{definition}\label{def:span}
  A \defining{linkspan}{span} \emph{\cite{benabou67,carboni87}} from $A$ to $B$,
  both objects of a category $\catC$, is a pair of morphisms with a common
  domain,
  \[A \overset{f}\longleftarrow E \overset{g}\longrightarrow B.\]
  The object $E$ is the
  ``head'' of the span, and the morphisms $f \colon E \to A$ and
  $g \colon E \to B$ are the left and right ``legs'', respectively.
\end{definition}

When the category $\catC$ has pullbacks, we can sequentially compose two spans %
$A \gets E \to B$ and $B \gets F \to C$ obtaining
$A \gets E \times_{B} F \to C$. Here, $E \times_{B} F$ is the pullback of $E$
and $F$ along $B$: for instance, in $\Set$,
$E \times_{B} F$ is the subset of $E \times F$ given by pairs that have the same image in $B$.

\begin{remark}[Notation for spans]
  We denote a span $A \overset{f}\gets X \overset{g}\to B$ in \(\catC\) as
  \[\gspan{f(x)}{x \in X}{g(x)} \in \Span(A,B),\]
  where $x \colon U \to X$, for some object \(U\) of \(\catC\), can be thought of as some generalized element that we compose with the two legs: e.g. in the category of sets, when $U = 1$, elements of a set $X$ can be seen as functions $x \colon 1 \to X$. Sometimes, these generalized elements will come
  with conditions that must be listed with the morphism set. For instance, in
  \Cref{fig:compositionofspans}, a composition of spans has a pullback as its
  head, so any generalized element of its head is now a pair of morphisms
  $x \colon U \to X$ and $y \colon U \to Y$ satisfying the extra condition \(g(x) = h(y)\):
\[\gspan{f(x)}{x}{g(x)} \comp \gspan{h(y)}{y}{k(y)} = \gspan{f(x)}{x,y}{k(y)}^{g(x) = h(y)}.\]

\begin{figure}[H]
  \centering
  \begin{tikzcd}[column sep=small, row sep=small]
    && X \times_B Y \drar{\pi_Y}\dlar[swap]{\pi_X} && \\
    & X \drar{g}\dlar[swap]{f} && Y \drar{k}\dlar[swap]{h} & \\
    A && B && C
  \end{tikzcd}
  \caption{Composition of spans.}
  \label{fig:compositionofspans}
\end{figure}

In other words, we are saying that the set of generalized elements of the head of the span is $\{ x,y \mid g(x) = h(y) \}$.
The advantage of this notation is that we can reason in any category with finite limits as we do in the category of sets: using \emph{elements}.
Whenever two sets of generalized elements of the head of a span are isomorphic, the Yoneda lemma~\cite{maclane78} provides an isomorphism between the heads.
That isomorphism makes the two spans equivalent when it commutes with the two legs.
\end{remark}

\begin{definition}
  Let $\catC$ be a category with pullbacks.
  \defining{linkspan}{$\mathbf{Span}(\catC)$} is the category that has the same
  objects as $\catC$ and isomorphism classes of spans between them as morphisms.
  That is, two spans are considered \emph{equal} if there is an isomorphism
  between their heads that commutes with both legs. Dually, if $\catC$ is a
  category with pushouts, \defining{linkcospan}{$\mathbf{Cospan}(\catC)$} is the
  category $\Span(\catC^{op})$.
\end{definition}

$\Span(\catC)$ is a \symmetricMonoidalCategory{} when $\catC$ has products.
The parallel composition of
$\gspan{f_{1}(x)}{x \in X}{g_{1}(x)} \in \Span(A_{1}, B_{1})$ and
$\gspan{f_{2}(y)}{y \in Y}{g_{2}(y)} \in \Span(A_{2}, B_{2})$ is given by the componentwise product
\[\gspan{(f_{1}(x),f_{2}(y))}{x \in X, y \in Y}{(g_{1}(x), g_{2}(y))} \in \Span(A_{1} \times A_{2}, B_{1} \times B_{2}).\]
An example is again $\Span(\Set)$.

\begin{remark}[Variable change]
  We will be considering spans ``up to isomorphism of their head''. This means
  that, given any isomorphism $\phi \colon X \to Y$, the following two spans are
  considered equal
  \[\gspan{f(\phi(x))}{x \in X}{g(\phi(x))} = \gspan{f(y)}{y \in Y}{g(y)}.\]
  Moreover, if two spans are equal, then such a variable change does necessarily exist.
\end{remark}

\begin{example}\label{ex:nor-latch-intro}
Let us now detail some useful constants of the algebra of $\Span(\catC)$,
which we will use to construct the NOR
latch circuit from \Cref{fig:norwithfeedback}.

  The Frobenius algebra \cite{carboni87} ($\blackComonoid$, $\blackMonoid$, $\blackMonoidUnit$,
  $\blackComonoidUnit$) is used for the ``wiring''. The following spans are constructed out of
  diagonals $A \to A \times A$ and units $A \to 1$.
  \begin{gather*}
   (\cpy)_{A} = \gspan{a}{a \in A}{(a,a)} \in \Span(A,A\times A) \qquad
   (\discard)_{A} = \gspan{a}{a}{\ast} \in \Span(A,1)\\
   (\cocpy)_{A} = \gspan{(a,a)}{a \in A}{a} \in \Span(A \times A,A)  \qquad
   (\codiscard)_{A} = \gspan{\ast}{a}{a} \in \Span(1,A)
 \end{gather*}
 These induce a compact closed structure (and thus a trace),
 as follows:
 \[\begin{aligned}
   (\cup)_{A} = \gspan{\ast}{a \in A}{(a,a)} \in \Span(1,A \times A) \\
   (\cap)_{A} =  \gspan{(a,a)}{a \in A}{\ast} \in \Span(A \times A,1).
 \end{aligned}\]
Finally, we have a braiding making the category symmetric,
\[(\swappingDiagram) = \gspan{(a,b)}{a \in A,b \in B}{(b,a)} \in \Span(A \times B, B \times A).\]
\end{example}

In general, any function $f \colon A \to B$ can be lifted  to a span $\gspan{a}{a \in A}{f(a)} \in \Span(A, B)$ covariantly, and to a span $\gspan{f(a)}{a \in A}{a} \in \Span(B, A)$, contravariantly.

\subsection{The Algebra of Open Transition Systems}\label{sec:algebratransition}
\begin{definition}\label{def:cat-graph}
  The category \defining{linkgraph}{$\mathbf{Graph}$} has graphs \(G=(s,t \colon E \rightrightarrows V)\) as objects, i.e.
  pairs of morphisms from \emph{edges} to \emph{vertices} returning the \emph{source} and \emph{target} of each edge.
  A morphism of graphs \((e,v) \colon G \to G'\) is given by functions \(e \colon E \to E'\) and \(v \colon V \to V'\) such that \(e \comp s' = s \comp v\) and \(e \comp t' = t \comp v\) (see \Cref{figure:morphismofgraphs})\footnote{Equivalently, $\Graph$ is the presheaf category on the diagram $(\bullet \rightrightarrows \bullet)$, i.e. the category of functors \((\bullet \rightrightarrows \bullet) \to \Set\) and natural transformations between them.}.
  \begin{figure}[H]
    \centering
    \begin{tikzcd}[column sep=large]
      E \rar{e} \dar[bend left]{t} \dar[bend right,swap]{s} & E' \dar[bend left]{t'} \dar[bend right,swap]{s'}\\
      V \rar{v} & V'
    \end{tikzcd}
    \caption{Morphism of graphs.}\label{figure:morphismofgraphs}
  \end{figure}
\end{definition}

We now focus on $\SpanGraphV$, those spans of graphs that have single vertex graphs $(A \rightrightarrows 1)$ as the boundaries.

\begin{definition}\label{def:open-transition-system}
  An \defining{linkopentransitionsystem}{open transition system} is a morphism of $\SpanGraphV$: a span of sets $\gspan{f(e)}{e \in E}{g(e)} \in \Span(A, B)$ where the head is the set of edges of a graph $s, t \colon E \rightrightarrows V$, i.e. the transitions (see \Cref{fig:morphismspangraph}). Two open transition systems are considered \emph{equal} if there is an isomorphism between their graphs that commutes with the legs.
  Open transition systems whose graph $E \rightrightarrows 1$ has a single vertex are called \emph{stateless}.
\begin{figure}[H]
  \centering
      \begin{tikzcd}
        A \dar[bend left]{} \dar[bend right,swap]{} &
        E \dar[bend left]{t} \dar[bend right,swap]{s} \lar[swap]{f} \rar{g} &
        B \dar[bend left] \dar[bend right,swap] \\
        1 & V \rar \lar & 1
      \end{tikzcd}
  \caption{A morphism of $\protect\SpanGraphV$.}\label{fig:morphismspangraph}
\end{figure}
\end{definition}
Sequential composition (the \emph{communicating-parallel operation} of~\cite{katis97}) of two open transition systems with spans
\[\gspan{f(e)}{e \in E}{g(e)} \in \Span(A, B)\mbox{ and }\gspan{h(e')}{e' \in E'}{k(e')} \in \Span(B, C)\]
and graphs $s,t \colon E \rightrightarrows S$ and $s', t' \colon E' \rightrightarrows S'$ yields the open transition system with the composite span
\[\gspan{f(e)}{(e,e') \in E \times E'}{k(e')}^{g(e)=h(e')} \in \Span(A, C)\]
and graph $(s \times s', t \times t') \colon E \times_{B} E' \rightrightarrows S \times S'$. This means that the only allowed transitions are those that synchronize $E$ and $E'$ on the common boundary $B$.

Parallel composition (the \emph{non communicating-parallel operation} of %
\cite{katis97}) of two open transition systems with spans
\[\gspan{f(e)}{e \in E}{g(e)} \in \Span(A, B)\mbox{ and }\gspan{f'(e')}{e' \in E'}{g'(e')} \in \Span(A', B')\]
and graphs $s,t \colon E \rightrightarrows V$ and
$s',t' \colon E' \rightrightarrows V'$ yields the open transition system with span
\[\gspan{(f(e),f'(e'))}{e,e' \in E \times E'}{(g(e), g(e'))} \in \Span(A \times A', B \times B')\]
and graph $(s \times s', t \times t') \colon E \times E' \rightrightarrows V \times V'$.

\label{sec:componentsspangraph}

\begin{remark}[Components of $\SpanGraphV$]
Any span in $\Span(A, B)$ can be lifted to $\SpanGraphV(A,B)$ by making the head represent the graph $E \rightrightarrows 1$. Apart from the components lifted from $\Span{}$, which we call \emph{stateless}, we will need to add a single stateful component to model all of $\SpanGraphV$: the delay in \Cref{figure:delaymorphism}.

\begin{figure}[H]
  \centering
  \examplesDelay
  \caption{Delay morphism over the set $\Bool \coloneqq \{0,1\}$.}
  \label{figure:delaymorphism}
\end{figure}

The delay $(\delayBox_{A})$ on a given set $A$ is given by the span $\gspan{a_{2}}{a_{1}, a_{2} \in A \times A}{a_{1}} \in \Span(A,A)$ together with the graph $\pi_{1},\pi_{2} \colon A \times A \to A$. This is to say that the delay receives on the left what the target of its transition (its next state) will be, while signalling on the right what the source of its transition (its current state) is.
This is \emph{not} an arbitrary choice: it is defined as the canonical delay obtained from
the feedback structure in $\SpanGraphV$ (as in \Cref{section:preliminaries},
$\delay_{A} = \mathsf{fbk}(\sigma_{A,A})$). %

\[(\delayBox)_{A} = \fbk(\gspan{(a_{1},a_{2})}{a_{1},a_{2}}{(a_{2},a_{1})}).\]

\end{remark}
We can use this delay to correctly model a stateful NOR latch from the function $\NOR \colon \Bool \times \Bool \to \Bool$ (as we saw in \Cref{fig:norandlatch}). %

\begin{figure}[H]
  \centering
  \norGatesSplit
  \caption{Decomposing the circuit.}
  \label{fig:nor-latch-decomposed}
\end{figure}

The NOR latch circuit of \Cref{fig:norwithfeedback} is the composition of two NOR gates where the outputs of each gate have been copied and fed back as input to the other gate.
The algebraic expression, in \(\SpanGraphV\), of this circuit is obtained by decomposing it into its components, as in \Cref{fig:nor-latch-decomposed}.
\begin{multline*}
  (\im \tensor \cup \tensor \cup \tensor \im) \comp
  (\NOR \tensor \sigma \tensor \NOR) \comp
  (\cpy \tensor \im \tensor \cpy)\\ \comp
  (\im \tensor \delay \tensor \im \tensor \delay \tensor \im) \comp
  (\im \tensor \cap \tensor \cap \tensor \im)
\end{multline*}

  The graph obtained from computing this expression, together with its
  transitions, is shown in \Cref{fig:graph-latch}. This time, our model is
  indeed stateful. It has four states: two states representing a correctly stored signal,
  $\LatchNotA = (1,0)$ and $\LatchA = (0,1)$; and two states representing
  transitory configurations $\LatchTOne = (0,0)$ and $\LatchTTwo = (1,1)$.
  \begin{figure}[H]
    \centering
    \graphLatch
    \caption{Span of graphs representing the NOR latch}
    \label{fig:graph-latch}
  \end{figure}

  The \emph{left boundary} can receive a \emph{set}
  signal, $\LatchSet = \binom{1}{0}$; a \emph{reset} signal,
  $\LatchReset = \binom{0}{1}$; none of the two, $\LatchIdle = \binom{0}{0}$; or
  both of them at the same time, $\LatchUnspec = \binom{1}{1}$, which is known
  to cause unspecified behavior in a NOR latch.
  The signal on the \emph{right boundary}, on the other hand, is always equal to
  the state the transition goes to and does not provide any additional
  information: we omit it from \Cref{fig:graph-latch}.

  \begin{figure}[H]
  \centering
  \[\fbkOn{}{\Bool \times \Bool}{}{} \left(\latchNoFeedback \right) \neq
  \tr_{\Bool \times \Bool} \left(\latchNoFeedback \right)\]
  \caption{Applying $\fbkOn{}{}{}{}(\bullet)$ over the circuit gives the NOR latch.}\label{fig:nor-latch-no-feedback}
  \end{figure}
  Activating the signal $\LatchSet$ %
makes the latch reach the state $\LatchA$ %
in (at most) two transition steps. Activating $\LatchReset$ %
does the same for $\LatchNotA$. After any of these two cases, deactivating all signals, $\LatchIdle$, keeps
the last state.

  Moreover, the (real-world) NOR latch has some unspecified behavior that gets also reflected in the graph: activating both $\LatchSet$ and $\LatchReset$ at the same time, what we call $\LatchUnspec$, causes the circuit to enter an unstable state where it bounces between the states $\LatchTOne$ and $\LatchTTwo$ after an $\LatchIdle$ signal. Our modeling has reflected this ``unspecified behavior'' as expected.

  \subsubsection{Feedback and trace.}
  In terms of feedback, the circuit of \Cref{fig:graph-latch} is equivalently
  obtained as the result of taking feedback over the stateless
  morphism in \Cref{fig:nor-latch-no-feedback}.

  But $\SpanGraphV$ is also canonically traced: it is actually compact closed.
  What changes in the modeling if %
  we would have used the
  trace instead? As we argued for \Cref{fig:nor-latch-trace}, we obtain
  a stateless transition system. %
  The valid transitions are
  \[\{
    (\LatchUnspec,\LatchTOne),
    (\LatchIdle,\LatchA),
    (\LatchIdle,\LatchNotA),
    (\LatchSet,\LatchA),
    (\LatchReset,\LatchNotA)
    \}.\]
  They encode important information: they are the \emph{equilibrium} states of
  the circuit. However, unlike the previous graph, this one would not get us the correct  allowed transitions: under this modeling, our circuit could freely bounce between
  $(\LatchIdle,\LatchA)$ and $(\LatchIdle,\LatchNotA)$, which is not the
  expected behavior of a NOR latch.

  The fundamental piece making our modeling succeed the previous time was
  feedback with \emph{delay}. Next we show that this feedback is canonical.

\subsection{Span(Graph) as a Feedback Category}
This section presents our main theorem. We introduce a mapping
that associates to each stateful span of sets a corresponding span of graphs.
This mapping is well-defined and lifts to a functor
\(\St(\Span(\Set)) \to \Span(\Graph)\). Finally, we prove that it is an
isomorphism \(\St(\Span(\Set)) \cong \SpanGraphV\).

First of all, we need to be able to explicitly compute the composition of stateful spans, following the composition of stateful morphisms in \Cref{def:sequentialcomposition}. This is \Cref{lemma:composingstatefulspans}. Then, we will characterize isomorphisms in the category of spans in \Cref{prop:isospans}.

\begin{proposition}\label{lemma:composingstatefulspans}
Let $\catC$ be a category with finite limits.
Consider two stateful spans in the category $\St(\Span(\catC))$,
\begin{align*}
  \gspan{(\sigma(x),f(x))}{x \in X}{(\sigma'(x),g(x))} &\in \Span(S \times A, S \times B),\\
  \gspan{(\tau(y), h(y))}{y \in Y}{(\tau'(y),k(y))} &\in \Span(T \times B, T \times C).
\end{align*}
their composition is then given by the span
\[\gspan{(\sigma(x), \tau(y), f(x))}{(x,y) \in X \times Y}{(\sigma'(x), \tau'(y), k(y))}^{g(x) = h(y)} \in \Span(S \times T \times A, S \times T \times B),\]
where the head is $X \times_{B} Y$, the pullback of $g$ and $h$.
\end{proposition}
\begin{proof}
  Using the notation for spans, we apply the definition of sequential composition in a category of stateful processes
  (\Cref{def:sequentialcomposition}).
  \begin{align*}
    & (\gspan{(s,t)}{s,t}{(t,s)} \otimes \gspan{a}{a}{a}) \comp
     (\gspan{t}{t}{t} \otimes \gspan{(\sigma(x), f(x))}{x}{(\sigma'(x), g(x))}) \\
    & \qquad \comp (\gspan{(t,s)}{s,t}{(s,t)} \otimes \gspan{b}{b}{b}) \comp
      (\gspan{s}{s}{s} \otimes \gspan{(\tau(y), h(y))}{y}{(\tau'(y), k(y))}) \\
    = & \quad \emph{(Computing tensors)} \\
    & \gspan{(s,t,a)}{s,t,a}{(t,s,a)} \comp \gspan{(t,\sigma(x), f(x))}{t,x}{(t,\sigma'(x), g(x))} \\
    & \qquad \comp \gspan{(t,s,b)}{s,t,b}{(s,t,b)} \comp \gspan{(s, \tau(y), h(y))}{s,y}{(s,\tau'(y), k(y))} \\
    = & \quad \emph{(Computing compositions)} \\
    & \gspan{(\sigma(x), t, f(x))}{t,x}{(t, \sigma'(x), g(x))} \comp \gspan{(\tau(y), s, h(y))}{s,y}{(s, \tau'(y), k(y))} \\
    = & \quad \emph{(Computing compositions)} \\
    & \gspan{(\sigma(x), \tau(y), f(x))}{x,y}{(\sigma'(x), \tau'(y),  k(y))}^{g(x) = h(y)}.
  \end{align*}

  This last formula corresponds indeed to the pullback we stated.
\end{proof}

\begin{proposition}\label{prop:isospans}
  Let $\catC$ be a category with all finite limits.
  An isomorphism $A \cong B$ in its category of spans, $\Span(\catC)$, is always of the form
  \[\gspan{a}{a \in A}{\phi(a)} \in \Span(\catC)(A,B),\]
  where the left leg is an identity and the right leg $\phi \colon A \to B$ is an isomorphism in the base category $\catC$.
\end{proposition}
\begin{proof}
Let $\gspan{f(x)}{x \in X}{g(x)} \in \Span(A,B)$ and $\gspan{h(y)}{y\in Y}{k(y)} \in \Span(B,A)$ be mutual inverses.
This means that
\[\gspan{f(x)}{x \in X}{g(x)} \comp \gspan{h(y)}{y \in Y}{k(y)} = \gspan{f(x)}{x,y}{k(y)}^{g(x) = h(y)} = \gspan{a}{a \in A}{a},\]
\[\gspan{h(y)}{y \in Y}{k(y)} \comp \gspan{f(x)}{x \in X}{g(x)} = \gspan{h(y)}{x,y}{g(x)}^{k(y) = f(x)} = \gspan{b}{b \in B}{b}.\]
In turn, this implies the existence of variable changes $(\alpha_{X},\alpha_{Y}) \colon A \to X \times_{B} Y$ and $(\beta_{Y}, \beta_{X}) \colon B \to Y \times_{A} X$ such that they are the inverses of $(f,k)$ and $(g,h)$ respectively.

We can thus write the spans as having the identity on the left leg.
\[\gspan{f(x)}{x \in X}{g(x)} = \gspan{f(\alpha_{X}(a))}{a \in A}{g(\alpha_{X}(a))} = \gspan{a}{a \in A}{g(\alpha_{X}(a))}.\]
\[\gspan{h(y)}{y \in Y}{k(y)} = \gspan{h(\beta_{Y}(b))}{b \in B}{k(\beta_{Y}(b))} = \gspan{b}{b \in B}{k(\beta_{Y}(b))}.\]
Finally, composing them again, we get that $(g \comp \alpha_{X})$ and $(k \comp \beta_{Y})$ must be mutual inverses, thus isomorphisms.
\end{proof}

We are now ready to prove the main result. The following \Cref{lemma:spanwelldefined} proves that we can translate stateful spans to spans of graphs. The main \Cref{theorem:spangraphfreefeedback} follows from it.

\begin{lemma}\label{proof:lemma-spanwelldefined}\label{lemma:spanwelldefined}
  Let $\catC$ be a category with all finite limits.
  The following assignment of \statefulProcesses{} over $\Span(\catC)$ to morphisms of $\Span(\Graph(\catC))$ is well-defined.
  \[K \left( \sS \;\middle|\;
      \begin{tikzcd}[column sep=tiny]
        & E \dlar[swap]{(s,f)}\drar{(t,g)} & \\
        S \times A & & S \times B \end{tikzcd}
    \right)
    \coloneqq
    \left( \begin{tikzcd}
        A \dar[bend left]{} \dar[bend right,swap]{} &
        E \dar[bend left]{t} \dar[bend right,swap]{s} \lar[swap]{f} \rar{g} &
        B \dar[bend left] \dar[bend right,swap] \\
        1 & S \rar \lar & 1
      \end{tikzcd}\right)\]
\end{lemma}

\begin{proof}
  We first check that two \emph{isomorphic} spans are sent to \emph{isomorphic} spans of graphs.
  Let
  \[\begin{aligned}
      \gspan{(s(e),f(e))}{e \in E}{(t(e),g(e))} &\in \Span(\sS \times \sA, \sS \times \sB)\mbox{ and } \\
      \gspan{(s'(e'),f'(e'))}{e' \in E'}{(t'(e'),g'(e'))} &\in \Span(\sS \times \sA, \sS \times \sB)
    \end{aligned}\]
  be two spans that are isomorphic with the variable change $\phi \colon E \cong E'$.
  Then, \((\phi, \mathrm{id})\) is an isomorphism of spans of graphs, also making the relevant diagram commute
  (\Cref{fig:isospansisographs}).

  \begin{figure}[H]
    \centering
    \begin{tikzcd}[row sep=tiny]
      & E \ar{dl}[swap]{(s,f)} \ar{dr}{(t,g)} \ar{dd}{\phi} & \\
      \sS \times \sA &  & \sS \times \sB \\
      & E' \ular{(s',f')} \urar[swap]{(t',g')}&
    \end{tikzcd} \qquad
    \begin{tikzcd}
      A \dar[bend left]{} \dar[bend right,swap]{}
      & E \rar{\phi} \lar[swap]{f} \dar[bend left]{t} \dar[bend right,swap]{s} \ar[bend left]{rr}{g}
      & E' \dar[bend left]{t'} \ar[bend right,swap]{ll}{f'} \dar[bend right,swap]{s'} \rar{g'}
      & B \dar[bend left]{} \dar[bend right,swap]{} \\
      1 & S \rar{\im} \lar{} \ar[bend right,swap]{rr}{}
      & S \rar \ar[bend left,swap]{ll}{}
      & 1
    \end{tikzcd}
    \caption{Isomorphic spans result in isomorphic spans of graphs.}
    \label{fig:isospansisographs}
  \end{figure}

  We show now that the assignment preserves the equivalence relation of \statefulProcesses{}.
  Isomorphisms in a category of spans are precisely spans whose two legs are isomorphisms (by \Cref{prop:isospans}, or the more general result of~\cite{pavlovic95}).
  This means that an isomorphism in \(\Span(\Set)\) can be always rewritten as \(\gspan{s}{s \in \sS}{\phi(s)} \in \Span(\sS, \sT)\), where the left leg is an identity and the right leg is \(\phi \colon \sS \to \sT\), some isomorphism.
  Its inverse can be written analogously as \(\sT \gets \sS \to \sS\).
  In order to prove that the quotient relation induced by the feedback is preserved, we need to check that equivalent spans of sets are sent to isomorphic spans of graphs.
  If two spans are equivalent with the variable change \(\phi \colon S \cong T\), then the corresponding graphs are isomorphic with the isomorphism of graphs \((\im,\phi)\), see \Cref{fig:equivalentspans}.

  \begin{figure}[H]
  \centering
    \begin{tikzcd}[column sep=-1ex,baseline=(B.base)]
      & \cdot \ar[dash]{dr} \ar{dl}[swap]{\phi \times \im} &&
      E \ar{dl}[swap]{(s,a)} \ar{dr}{(t,b)} &&
      \cdot \ar{dr}{\phi \times \im} \ar[dash]{dl} &\\
      \sT \times \sA &&
      \sS \times \sA &&
      \sS \times \sB  &&
      |[alias=B]| \sT \times \sB \\
    \end{tikzcd} \qquad
    \begin{tikzcd}[baseline=(B.base)]
      A \dar[bend left]{} \dar[bend right,swap]{}
      & E \rar{\im}
      \lar[swap]{a}
      \dar[bend left]{t}
      \dar[bend right,swap]{s}
      \ar[bend left]{rr}{b}
      & E \dar[bend left]{t \comp \phi} \ar[bend right,swap]{ll}{a}
          \dar[bend right,swap]{s\comp \phi} \rar{b}
      & B \dar[bend left]{} \dar[bend right,swap]{} \\
      1 & |[alias=B]| S \rar{h} \lar{} \ar[bend right,swap]{rr}{}
      & T \rar \ar[bend left,swap]{ll}{}
      & 1
    \end{tikzcd}
    \caption{Equivalent spans result in isomorphic spans of graphs.}
    \label{fig:equivalentspans}
  \end{figure}
\end{proof}

\begin{theorem}\label{theorem:spangraphfreefeedback}
  There exists an isomorphism of categories
  $\St(\Span(\Set)) \cong \SpanGraphV$.
  That is, the free feedback category over \(\Span(\Set)\) is isomorphic to the full subcategory of \(\Span(\Graph)\) given by single-vertex graphs.
\end{theorem}
\begin{proof}
  We prove that there is a fully faithful functor
  \(K \colon \St(\Span(\Set)) \to \Span(\Graph)\) defined on objects as
  \(K(A) = (A \rightrightarrows 1)\) and defined on morphisms as in \Cref{lemma:spanwelldefined}.

  We now show that $K$ is functorial, preserving composition and identities.
  The identity morphism on \(A\) in \(\St(\Span(\Set))\) has state space \(1\), so it is a span $1 \times A \gets A \to 1 \times A$ and it is sent to the identity span on the graph $A \rightrightarrows 1$.

  Composition is also preserved.
  Let us consider two stateful spans
  \begin{align*}
    \gspan{(s(e),a(e))}{e \in E}{(s'(e),b(e))} &\in \Span(\sS \times \sA, \sS \times \sB)\mbox{ and } \\
    \gspan{(t(f),b'(f))}{f \in F}{(t'(f),c(f))} &\in \Span(\sS \times \sB, \sS \times \sC)
  \end{align*}
  By \Cref{lemma:composingstatefulspans}, their composition is given by the span
  \[\gspan{(s(e),t(f),a(e))}{(e,f) \in E \times F}{(s'(e),t'(f),c(f))}^{b(e)= b'(f)} \in \Span(\sS \times \sT \times \sA, \sS \times \sT \times \sC)\]
  where the head $E \times_{B} F$ is the pullback of $b$ and $b'$.

  \begin{figure}[H]
  \centering
  \begin{tikzcd}[column sep=small]
    && E \times_{B} F \dlar[swap]{\pi_{E}} \drar{\pi_{F}} \dar[xshift=0.5ex] \dar[xshift=-0.5ex] && \\
    & E \dlar[swap]{a} \drar[swap,near end]{b} \dar[xshift=0.5ex] \dar[xshift=-0.5ex]
    & S \times T \dlar[swap,near start]{\pi_{S}} \drar[near start]{\pi_{T}}
    & F \drar{c} \dlar[near end]{b'} \dar[xshift=0.5ex] \dar[xshift=-0.5ex]
    & \\
    A \dar[xshift=0.5ex] \dar[xshift=-0.5ex]
    & S \drar \dlar
    & B \dar[xshift=0.5ex] \dar[xshift=-0.5ex]
    & T \drar \dlar
    & C \dar[xshift=0.5ex] \dar[xshift=-0.5ex] \\
    1 && 1 && 1
  \end{tikzcd}
  \caption{Pullback of graphs.}\label{fig:graphpullback}
  \end{figure}

  We have composed two stateful spans and we want to show that the graph corresponding to their composition is the pullback of the graphs corresponding to them.
  Computing a pullback of graphs can be done separately on edges and vertices, as graphs form a presheaf category (see \Cref{fig:graphpullback}).
  Note how the resulting graph is precisely the graph corresponding, under the assignment $K$, to the stateful span computed above.

  We have shown that \(K\) is a functor.
  The final step is to show that it is fully-faithful.
  We can see that it is full: every span of single-vertex graphs given by $\gspan{f(e)}{e \in E}{g(e)} \in \Span(A, B)$ and $s, t \colon E \rightrightarrows S$ is the image of some span, namely
  \[\gspan{s(e),f(e)}{e \in E}{t(e),g(e)} \in \Span(S \times A, S \times B).\]
  Let us check it is also faithful.
  Suppose that two morphisms in \(\St(\Span(\Set))\), \(\sS \times \sA \gets E \to \sS \times \sB\) and \(\sS' \times \sA \gets E' \to \sS' \times \sB\), are sent to equivalent spans of graphs, i.e. there exist \(h \colon E \cong E'\) and \(k \colon S' \cong S\) making the diagrams in \Cref{fig:equivalentspangraph} commute.

\begin{figure}[H]
   \[\begin{tikzcd}
    A \dar[bend left] \dar[bend right]
    & E \dar[bend left]{t} \dar[bend right,swap]{s} \rar{h} \ar[swap]{l}{a} \ar[bend left]{rr}{b}
    & E' \dar[bend left]{t'} \ar[bend right,swap]{d}{s'} \ar[bend right,swap]{ll}{a'} \ar{r}{b'}
    & B \dar[bend left] \dar[bend right]
    \\ 1
    & S \rar{k} \ar[bend right]{rr} \ar{l}
    & S' \ar[bend left]{ll} \ar{r}
    & 1
  \end{tikzcd}\]
  \caption{Equivalent spans of graphs.}\label{fig:equivalentspangraph}
\end{figure}

  The isomorphism \(k\) makes the following spans equivalent as \statefulProcesses{}.
  \begin{align*}
    S \times A \gets & E \to S \times B\\
    S' \times A \gets & E \to S' \times B
  \end{align*}
  Moreover, the isomorphism \(h\) makes the following spans equivalent as spans, showing faithfullness of \(K\).
  \begin{align*}
    S' \times A \gets & E \to S' \times B\\
    S' \times A \gets & E' \to S' \times B
  \end{align*}

  We have shown that there exists a fully-faithful functor from the free feedback category over \(\Span(\Set)\) to the category \(\Span(\Graph)\) of spans of graphs.
  The functor restricts to an equivalence between \(\St(\Span(\Set))\) and the full subcategory of \(\Span(\Graph)\) on single-vertex graphs.
  It is moreover bijective on objects, giving an isomorphism of categories.
\end{proof}

\begin{example}
The characterization $\SpanGraphV \cong \St(\Span(\Set))$ that we prove in
  \Cref{theorem:spangraphfreefeedback} lifts the inclusion $\Set \to \Span(\Set)$ to a
  feedback preserving functor $\mathbf{Mealy} \to \SpanGraphV$. This inclusion embeds a
  deterministic transition system into the graph that determines it.
\end{example}

\begin{example}\label{ex:non-equivalent-automata}
  Following from \Cref{rem:sliding-isos}, we present an example of spans of graphs that would be equated if we assumed the sliding axiom~\ref{axiom:slide} of feedback categories for arbitrary morphisms rather than just isomorphisms.
  Consider the spans of sets \(\alpha \colon W \to V\) and \(h \colon V \to W \times \Bool\) as in \Cref{fig:non-equivalent-automata}, where $V = \{v_1, v_2\}$ and $W = \{w\}$.
  Depending on where the feedback operation is applied, we obtain two different spans of graphs, \(g = \fbk_V(h \comp (\alpha \tensor \im))\) and \(f = \fbk_W((\alpha \tensor \im) \comp h)\): the first one contains an additional transition. If we were to impose that the sliding axiom holds for non-isomorphisms, we could erase this additional transition, and obtain that \(f = g\) by sliding \(\alpha\) through the feedback loop.
\end{example}

\begin{figure}[h!]
  \noneqSpansGraph{}
  \caption{Spans of graphs that would be equated by a stronger notion of equivalence: \(g = \fbk(h \comp (\alpha \tensor \im)) \sim \fbk((\alpha \tensor \im) \comp h) = f\).}\label{fig:non-equivalent-automata}
\end{figure}

\subsection{Cospan(Graph) as a Feedback Category}

\Cref{theorem:spangraphfreefeedback} can be generalized to any category \(\catC\) with finite limits, where we can define graphs and spans of them.

A graph in a category \(\catC\) is given by two objects, \(E\) and \(V\), and two morphisms in \(\catC\), the source and the target \(s,t \colon E \to V\).
A morphism of graphs \(\alpha \colon G \to G'\) in \(\catC\) is a pair of morphisms, \(\alpha_{E} \colon E \to E'\) and \(\alpha_{V} \colon V \to V'\), in \(\catC\) that commute with the sources and the targets.
In categorical terms, these can be reformulated as functors and natural transformations.

\begin{definition}
  Let \(\catC\) be a category with finite limits.
  A graph in \(\catC\) is a functor from the diagram $(\bullet \rightrightarrows \bullet)$ to \(\catC\).
  A morphism of graphs in \(\catC\) is a natural transformation between the corresponding functors.
  Graphs in \(\catC\) form a category $\Graph(\catC)$.
\end{definition}

Categories of functors into \(\catC\) have all the limits that \(\catC\) has~\cite{maclane78}.
We can then form the category \(\Span(\Graph(\catC))\) and take its full subcategory on objects of the form $A \rightrightarrows 1$, i.e. ${\Span(\Graph(\catC))}_{\ast}$, to obtain:

\begin{theorem}\label{theorem:spangraphfreefeedbackarbitraryc}
  There exists an isomorphism of categories
  $\St(\Span(\catC)) \cong \Span(\Graph(\catC))_{\ast}$.
  That is, the free feedback category over \(\Span(\catC)\) is equivalent to the full
  subcategory on \(\Span(\Graph(\catC))\) given by single-vertex graphs.
\end{theorem}

\label{sec:cospangraphfeedback}
\(\CospanGraphE\) is the dual algebra to \(\SpanGraphV\).
Its morphisms represent graphs with discrete boundaries:
while, in \(\SpanGraphV\), each transition in the graph is  assigned a boundary behavior, a morphism in \(\CospanGraphE\) is a graph where some vertices are marked as left boundary or right boundary vertices.
This allows graphs to be composed by identifying these boundary vertices.

\begin{definition}\label{def:graph-discrete-boundaries}
  A \emph{graph with discrete boundaries} \(g \colon X \to Y\) is given by a graph \(G = (s,t \colon E \rightrightarrows V)\) and two functions, \(l \colon X \to V\) and \(r \colon Y \to V\), marking the boundary vertices.
\end{definition}

\begin{example}\label{ex:cospangraph-composition}
  We represent the legs of a cospan as dashed arrows pointing to some vertices of the apex graph.
  \[\cospangraphExFig{}\]
  The composition of the above cospans of graphs is given by
  \[\cospangraphCompositionExFig\quad,\]
  where the vertices in the common boundary have been identified.
\end{example}

$\CospanGraphE$ can be also characterized as a free feedback category. We
know that \(\Cospan(\Set) \cong \Span(\Set^{op})\), we note that
$\Graph(\Set^{op}) \cong \Graph^{op}(\Set)$ (which has the effect of flipping edges and
vertices), and we can use \Cref{theorem:spangraphfreefeedbackarbitraryc} because
$\Set$ has all finite colimits. The explicit assignment is similar to the one
shown in \Cref{lemma:spanwelldefined}.
  \[K \left( \sS \;\middle|\; \begin{tikzcd}[column sep=tiny]
        & S  & \\
        E + A \urar{[t\mid a]} & & E + B  \ular[swap]{[s \mid b]}\end{tikzcd}\right)
    \coloneqq
    \left( \begin{tikzcd}
        A \rar{a} &
        S &
        B  \lar[swap]{b} \\
        0  \uar[bend left] \uar[bend right,swap] \rar
        & E  \uar[bend right,swap]{s} \uar[bend left]{t}
        & 0  \uar[bend left] \uar[bend right,swap]  \lar
      \end{tikzcd}\right)\]

\begin{corollary}
  There is an isomorphism
  \[\St(\Cospan(\Set)) \cong \CospanGraphE.\]
\end{corollary}

\(\Cospan(\Graph)\) is also compact closed and, in particular, traced. As in the case of \(\Span(\Graph)\), the feedback structure given by the universal property is different from the trace. In the case of \(\Cospan(\Graph)\), tracing has the effect of identifying the output and input vertices of the graph; while feedback adds an additional edge from the output to the input vertices.
\begin{example}
  Tracing the cospan of a one-edge graph identifies the two vertices making it into a self-loop.
  On the other hand, taking feedback of the same cospan has the effect of adding another edge from the right boundary to the left one.
  \[\traceFbkCospanExFig{}\]
\end{example}

\subsection{Syntactical Presentation of Cospan(FinGraph)}\label{section:applications}
The observation in \Cref{proposition:delaycompactclosed} has an important consequence in the case of finite sets.
We write $\mathbf{Fin}\Graph$ for $\Graph(\FinSet)$.
$\Cospan(\FinSet)$ is the generic special commutative Frobenius algebra~\cite{lack04}, meaning that any morphism written out of the operations of a special commutative Frobenius monoid and the structure of a symmetric monoidal category is precisely a cospan of finite sets.
\Cref{fig:frobunnyus} represents the generators and the axioms of the generic special commutative Frobenius monoid.
\begin{figure}[h!]
  \frobunnyusAxioms{}
  \caption{Generators and axioms of the generic special commutative Frobenius monoid.}\label{fig:frobunnyus}
\end{figure}

But we also know that $\Cospan(\FinSet)$, with an added generator to its PROP structure~\cite{bonchi19}
is $\St(\Cospan(\FinSet))$, or, equivalently, $\Cospan(\mathbf{Fin}\Graph)$.
This means that any morphism written out of the operations of a special commutative Frobenius algebra plus a freely added generator of type $(\delayBox) \colon 1 \to 1$ is a morphism in ${\Cospan(\mathbf{Fin}\Graph)}_{\ast}$.
This way, we recover one of the main results of~\cite{rosebrugh05} as a direct corollary of our characterization.

\begin{proposition}[Proposition 3.2 of \cite{rosebrugh05}]
  ${\Cospan(\FinGraph)}_{\ast}$ is the generic special commutative Frobenius
  monoid with an added generator.
\end{proposition}
\begin{proof}
  It is known that the category $\Cospan(\FinSet)$ is the generic special
  commutative Frobenius algebra \cite{lack04}.
  Adding a free generator $(\delayBox) \colon 1 \to 1$ to its PROP structure
  corresponds to adding a family $(\delayBox)_{n} \colon n \to n$ with the conditions
  on \Cref{proposition:delaycompactclosed}. Now,
  \Cref{proposition:delaycompactclosed} implies that adding such a generator to
  $\Cospan(\FinSet)$ results in $\St(\Cospan(\FinSet))$.  Finally, we
  use \Cref{theorem:spangraphfreefeedback} to conclude that
  $\St(\Cospan(\FinSet)) \cong \Cospan(\FinGraph)_{\ast}$.
\end{proof}
\begin{example}
  The delay generator \(\delayBox \colon 1 \to 1\) in \(\Cospan(\FinGraph)_{\ast}\) can be interpreted as a single edge.
  Thus, we draw it as \(\textedgepic \colon 1 \to 1\).
  The cospans of graphs in \Cref{ex:cospangraph-composition} are represented by the string diagrams
  \[\cospangraphStringExFig{}\quad.\]
  Their composition is
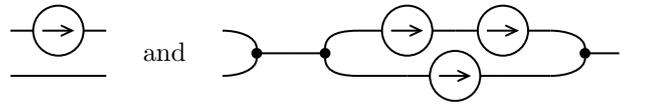
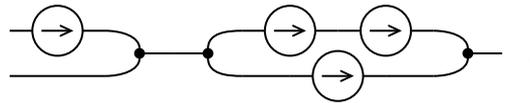
  \[\cospangraphCompositionStringExFig{}\quad.\]
\end{example}
 \newpage

\section{Structured state spaces}\label{sec:generalizingfeedback}
This section extends the framework of feedback categories from mere transition systems to automata with initial and final states.
In order to achieve this, we generalize the feedback construction to deal with a richer structure.
The key ingredient in the generalization of feedback categories is a close examination of the sliding axiom: deciding which processes can be ``slid'' determines the notion of equality we want to apply.

\subsection{Structured Feedback Categories}
In order to capture automata, the state space \(\sS\) needs to be equipped with ``extra structure''.
This is achieved by letting the state space live in a different category \(\catS\) from the base category \(\catC\) and by having a way of ``forgetting'' the extra structure it carries through a monoidal functor \(\Rm \colon \catS \to \catC\).

A structured feedback operator on \(\catC\), then, takes a morphism acting on a structured state space (\Cref{fig:structuredfeedbackcategory}).

\begin{figure}[H]
  \centering
  \(\infer{\fbkOn{}{\sS}{\sA}{\sB}(\fm) \colon \sA \to \sB}{f \colon \Rm\sS \otimes \sA \to \Rm\sS \otimes \sB}\)
  \caption{Type of the operator $\fbkOn{}{\sS}{\sA}{\sB}(\bullet)$.}\label{fig:structuredfeedbackcategory}
\end{figure}

\begin{definition}[Structured feedback category]
  A \defining{linkcategorywithstructuredfeedback}{structured feedback category} is a
  symmetric monoidal category \((\catC,\tensor,I,\alpha^{\catC}, \lambda^{\catC}, \rho^{\catC})\) together with a symmetric monoidal category, \((\catS,\boxtimes, J, \alpha^{\catS}, \lambda^{\catS}, \rho^{\catS})\), representing structured state spaces, and a symmetric monoidal functor \((\Rm, \varepsilon, \mu) \colon \catS \to \catC\) endowed with an operator \(\fbkOn{}{\sS}{A}{B} \colon \catC(\Rm\sS \otimes A, \Rm\sS \otimes B) \to \catC(A,B)\), which satisfies the following axioms (B1-B5).
\addtolength\leftmargini{2em}
\begin{enumerate}[label={\quad(B\arabic*).}, ref={(B\arabic*)}, start=1 ]
  \item\label{axiom:Gtight} \defining{linkstructuredtighteningaxiom}{Tightening}. For every \(\sS \in \catS\), every morphism \(\fm \colon \Rm\sS \otimes \sA \to \Rm\sS \otimes \sB\) and every pair of morphisms \(\um \colon A' \to A\) and \(\vm \colon \sB \to B'\),
    \[\um \comp \fbkOn{}{\sS}{\sA}{\sB}(\fm) \comp \vm = \fbkOn{}{\sS}{\sA'}{\sB'}((\im \otimes \um) \comp \fm \comp (\im \otimes \vm)).\]
  \item\label{axiom:Gvanish} \defining{linkstructuredvanishingaxiom}{Vanishing}. For every $\fm \colon A \to B$,
    \[\fbkOn{}{J}{\sA}{\sB}((\inverse{\varepsilon} \tensor \im)\comp\fm\comp(\varepsilon\tensor\im)) = \fm.\]
  \item\label{axiom:Gjoin} \defining{linkstructuredjoiningaxiom}{Joining}. For every \(\sS, \sT \in \catS\) and every morphism $\fm \colon \Rm\sS \tensor \Rm\sT \tensor \sA \to \Rm\sS \tensor \Rm\sT \tensor \sB$,
    \[\fbkOn{}{\sT}{\sA}{\sB}(\fbkOn{}{\sS}{\sS \otimes \sA}{\sS \otimes \sB}(\fm)) = \fbkOn{}{\sS \otimes \sT}{A}{B}((\inverse{\mu}\tensor\im) \comp f \comp (\mu \tensor \im)).\]
  \item\label{axiom:Gstrength} \defining{linkstructuredstrengthaxiom}{Strength}. For every \(\sS \in \catS\), every morphism $\fm \colon \Rm\sS \tensor \sA \to \Rm\sS \tensor \sB$ and every morphism $\gm \colon \sA' \to \sB'$,
    \[\fbkOn{}{\sS}{\sA}{\sB}(\fm) \otimes \gm = \fbkOn{}{\sS}{\sA \otimes \sA'}{\sB \otimes \sB'}(\fm \otimes \gm).\]
  \item\label{axiom:Gslide} \defining{linkstructuredslidingaxiom}{Sliding}. For every \(\sS,\sT \in \catS\), every $\fm \colon \Rm\sT \tensor \sA \to \Rm\sS \tensor \sB$ and every $h \colon \sS \to \sT$ in \(\catS\),
    \[\fbkOn{}{\sT}{\sA}{\sB}(\fm \comp (\Rm\hm \otimes \im)) = \fbkOn{}{\sS}{\sA}{\sB}((\Rm\hm \otimes \im) \comp \fm).\]
\end{enumerate}
\end{definition}

\begin{remark}
  The sliding axiom encodes the fact that applying a transformation to the state space $h \colon S \to T$ just before computing the next state should be essentially the same as applying the same transformation to the state space just before retrieving the current state.  In the particular case where all transformations are asked to be reversible (i.e. isomorphisms)\footnote{Here, \(\catS\) is the subcategory of isomorphisms of \(\catC\) and \(R\) is the inclusion functor.}, this sliding axiom~\ref{axiom:Gslide} particularizes to the sliding axiom of plain feedback categories~\ref{axiom:slide}.
\end{remark}

\begin{definition}[Structured feedback functor]
  A \defining{linkstructuredfeedbackfunctor}{structured feedback functor} \((F,G) \colon \catC \to \catC'\) between two \hyperlink{linkcategorywithstructuredfeedback}{structured feedback categories} $(\catC,\catS,\Rm,\fbkOn{}{}{}{})$ and $(\catC',\catS',\Rm',\fbkIsi{}')$ is a pair of \hyperlink{linkmonoidalfunctor}{symmetric monoidal functors}, $(F,\varepsilon,\mu)$ and $(G,\varepsilon^{G},\mu^{G})$, with types \(F \colon \catC \to \catC'\) and \(G \colon \catS \to \catS'\) such that \(R \comp F = G \comp R'\) and
  \[F(\fbkOn{}{\sS}{\sA}{\sB}(\fm)) = \fbkIsi{G\sS}'(\mu \comp Ff \comp \mu^{-1}).\]
  We write \defining{linkcatstructuredfeedback}{$\mathsf{SFeedback}$} for the category of (small) \hyperlink{linkcategorywithstructuredfeedback}{structured feedback categories} and \hyperlink{linkstructuredfeedbackfunctor}{structured feedback functors}.
  There is a forgetful functor ${\cal U} \colon \SFEEDBACK \to \SYMMON$.
\end{definition}

\subsection{Structured $\mathsf{St}(\bullet)$ Construction}\label{sec:structured-st-construction}

In the same way that the free feedback category was realized by stateful processes, the free structured feedback category is realized by stateful processes with a structured state space $S \in \catS$. A functor $\Rm \colon \catS \to \catC$ forgets the extra structure of this space.

Following the analogy, stateful processes with structured state space are pairs
$(S \mid f)$ consisting of a structured state space $S \in \catS$ and a morphism
$f \colon \Rm S \tensor A \to \Rm S \tensor B$.
We shall consider morphisms up to sliding of their state space, as in \Cref{fig:structured-stateful-process}.

\begin{figure}[h!]
\centering
\begin{tikzpicture}[x=0.75pt,y=0.75pt,yscale=-1,xscale=1]
\draw   (130,60) -- (150,60) -- (150,100) -- (130,100) -- cycle ;
\draw    (110,90) -- (130,90) ;
\draw    (120,70) -- (124.83,70) -- (130,70) ;
\draw    (110,70) -- (120,70) ;
\draw    (150,90) -- (210,90) ;
\draw    (150,70) -- (170,70) ;
\draw  [color={rgb, 255:red, 155; green, 155; blue, 155 }  ,draw opacity=1 ][fill={rgb, 255:red, 155; green, 155; blue, 155 }  ,fill opacity=1 ] (107.46,70) .. controls (107.46,68.62) and (108.6,67.5) .. (110,67.5) .. controls (111.4,67.5) and (112.54,68.62) .. (112.54,70) .. controls (112.54,71.38) and (111.4,72.5) .. (110,72.5) .. controls (108.6,72.5) and (107.46,71.38) .. (107.46,70) -- cycle ;
\draw    (190,70) -- (210,70) ;
\draw   (170,60) -- (190,60) -- (190,80) -- (170,80) -- cycle ;
\draw  [color={rgb, 255:red, 155; green, 155; blue, 155 }  ,draw opacity=1 ][fill={rgb, 255:red, 155; green, 155; blue, 155 }  ,fill opacity=1 ] (207.46,70) .. controls (207.46,68.62) and (208.6,67.5) .. (210,67.5) .. controls (211.4,67.5) and (212.54,68.62) .. (212.54,70) .. controls (212.54,71.38) and (211.4,72.5) .. (210,72.5) .. controls (208.6,72.5) and (207.46,71.38) .. (207.46,70) -- cycle ;
\draw   (320,60) -- (340,60) -- (340,100) -- (320,100) -- cycle ;
\draw    (340,90) -- (360,90) ;
\draw    (270,70) -- (274.83,70) -- (280,70) ;
\draw    (260,70) -- (270,70) ;
\draw    (260,90) -- (320,90) ;
\draw    (300,70) -- (320,70) ;
\draw  [color={rgb, 255:red, 155; green, 155; blue, 155 }  ,draw opacity=1 ][fill={rgb, 255:red, 155; green, 155; blue, 155 }  ,fill opacity=1 ] (257.46,70) .. controls (257.46,68.62) and (258.6,67.5) .. (260,67.5) .. controls (261.4,67.5) and (262.54,68.62) .. (262.54,70) .. controls (262.54,71.38) and (261.4,72.5) .. (260,72.5) .. controls (258.6,72.5) and (257.46,71.38) .. (257.46,70) -- cycle ;
\draw    (340,70) -- (360,70) ;
\draw   (280,60) -- (300,60) -- (300,80) -- (280,80) -- cycle ;
\draw  [color={rgb, 255:red, 155; green, 155; blue, 155 }  ,draw opacity=1 ][fill={rgb, 255:red, 155; green, 155; blue, 155 }  ,fill opacity=1 ] (357.46,70) .. controls (357.46,68.62) and (358.6,67.5) .. (360,67.5) .. controls (361.4,67.5) and (362.54,68.62) .. (362.54,70) .. controls (362.54,71.38) and (361.4,72.5) .. (360,72.5) .. controls (358.6,72.5) and (357.46,71.38) .. (357.46,70) -- cycle ;
\draw (140,80) node    {$\fm$};
\draw (116.5,96.5) node  [font=\scriptsize]  {$\sA$};
\draw (116.5,56.5) node  [font=\scriptsize]  {$\Rm\sT$};
\draw (203.5,96.5) node  [font=\scriptsize]  {$\sB$};
\draw (161.5,56.5) node  [font=\scriptsize]  {$\Rm\sS$};
\draw (180,70) node    {$\Rm\hm$};
\draw (203.5,56.5) node  [font=\scriptsize]  {$\Rm\sT$};
\draw (226,73.4) node [anchor=north west][inner sep=0.75pt]    {$=$};
\draw (330,80) node    {$\fm$};
\draw (266.5,96.5) node  [font=\scriptsize]  {$\sA$};
\draw (266.5,56.5) node  [font=\scriptsize]  {$\Rm\sS$};
\draw (353.5,96.5) node  [font=\scriptsize]  {$\sB$};
\draw (311.5,56.5) node  [font=\scriptsize]  {$\Rm\sT$};
\draw (290,70) node    {$\Rm\hm$};
\draw (353.5,56.5) node  [font=\scriptsize]  {$\Rm\sS$};
\end{tikzpicture}
    \caption{Equivalence of \protect\structuredstatefulProcesses{}. We depict \protect\structuredstatefulProcesses{} by marking the state space.}\label{fig:structured-stateful-process}
  \end{figure}

\begin{definition}[Category of structured stateful processes]
  \nicelinktarget{linkGFbk}\nicelinktarget{linkstructuredfeedbackcircuit}
  Consider a pair of \hyperlink{linksymmetricmonoidalcategory}{symmetric monoidal categories} $(\catC,\tensor,I,\alpha^{\catC}, \lambda^{\catC}, \rho^{\catC})$ and \((\catS,\boxtimes,J,\alpha^{\catS}, \lambda^{\catS}, \rho^{\catS})\) and a symmetric monoidal functor \((\Rm, \varepsilon, \mu) \colon \catS \to \catC\).
  We write $\GSt(\catC,\Rm)$ for the category with the objects of $\catC$ but where morphisms $A \to B$ are pairs $(\sS \mid f)$ consisting of a state space $\sS \in \catS$ and a morphism $f \colon \Rm\sS \tensor A \to \Rm\sS \tensor B$.
  We consider morphisms up to equivalence of their state spaces, where the equivalence relation is generated by
  \[(\sS \mid (\Rm\hm \tensor \im) \comp f) \sim_{\catS} (\sT \mid f \comp (\Rm\hm \tensor \im))\mbox{ for any }h \colon \sS \to \sT.\]
Identities, composition, monoidal product and the feedback operator \(\defining{linkstructuredstore}{\ensuremath{\Gstore{}(\bullet)}}\) are defined in analogous ways as for \statefulProcesses{} (\Cref{def:sequentialcomposition}).
  When depicting a \structuredstatefulProcess{} (\Cref{fig:structured-stateful-process}) we mark the state strings.
\end{definition}

\begin{remark}
  In other words, structured stateful processes are elements of the following coproduct quotiented by the smallest equivalence relation $(\sim_{\catS})$ satisfying sliding: $((\Rm\hm \tensor \im) \comp f) \sim_{\catS} (f \comp (\Rm\hm \tensor \im))$.
  \begin{equation*}
    \GSt(\catC,\Rm)(X,Y) \coloneqq \left( \sum_{S \in \catS} \catC (\Rm\sS \tensor A, \Rm\sS \tensor B) \right) \big/ (\sim_{\catS}).
  \end{equation*}
  This quotient is a particular form of colimit called a \emph{coend}~\cite{maclane78}.
\end{remark}

Repeating the proof from Katis, Sabadini and Walters~\cite{katis97} in this generalized setting, we can show that structured stateful processes form the free structured feedback category.

\begin{theorem}\label{th:statestructured}
  The category \(\GSt(\catC,\Rm)\), endowed with the \(\Gstore{}(\bullet)\) operator, is the free structured feedback category over a \hyperlink{linksymmetricmonoidalcategory}{symmetric monoidal category} \(\catC\) with a symmetric monoidal functor \(\Rm \colon \catS \to \catC\).
\end{theorem}

\subsection{Categories of Automata}

Let us present classical automata as an example of the construction of structured feedback.
Automata have a structured state space where a particular state is considered the ``initial state'', and a subset of states are considered ``final''.  We can canonically recover a suitable category of automata as the free feedback category over these structured spaces.

\begin{definition}
  \defining{linkautomatonstatespace}{}
  An \emph{automaton state space} $(S,i_{S},f_{S})$ is a finite set $S$ together with an \emph{initial state} $i_{S} \in S$ and a subset of final states $f_{S} \colon S \to 2$. The product of two automaton state spaces, $(S,i_{S},f_{S})$ and $(T,i_{T},f_{T})$, is the state space $(S \times T,(i_{S},i_{T}),f_{S} \wedge f_{T})$,
  where the final states are pairs of final states, $(f_{S} \wedge f_{T})(s,t) = f_{S}(s) \wedge f_{T}(t)$.
  A morphism of automaton state spaces $\alpha \colon (S,i_{S},f_{S}) \to (T,i_{T},f_{T})$ is a function $\alpha \colon S \to T$ such that $i_{S} \comp \alpha = i_{T}$ and $f_{S} = \alpha \comp f_{T}$. Automaton state spaces form a symmetric monoidal category, $\AutSt$.
\end{definition}

\begin{remark}
  As a consequence, an isomorphism of automata state spaces
  \[\alpha \colon (S,i_{S},f_{S}) \cong (T,i_{T},f_{T})\]
  is an isomorphism $\alpha \colon S \cong T$ such that $i_{S} \comp \alpha = i_{T}$ and $f_{S} = \alpha \comp f_{T}$. These form a subcategory $\mathbf{Iso}(\AutSt)$ with forgetful functors $\mathsf{U}_{\mathsf{Iso}} \colon \mathbf{Iso}(\AutSt) \to \mathbf{FinSet}$ and $\mathsf{U}_{\mathsf{Aut}} \colon \mathbf{Iso}(\AutSt) \to \mathbf{Span}(\mathbf{FinSet})$.
\end{remark}

\begin{definition}
  \defining{linkmealydeterministicfiniteautomaton}
  A \emph{Mealy deterministic finite automaton}
  $(S,A,B,i_{S},f_{S},t_{S})$ is given by a finite set of states $S$, a finite
  alphabet of input symbols $A$ and a finite alphabet of output symbols $B$, an initial state $i_{S} \in S$, a set of final states $f_{S} \colon S \to 2$, and a transition function $t_{S} \colon S \times A \to S \times B$.
  The product of two deterministic finite automata,
  \[(S,A,B,i_{S},f_{S},t_{S})\mbox{ and }(S',A',B',i_{S'},f_{S'},t_{S'}),\]
  is the automaton $(S \times T,A,C,(i_{S},i_{S'}),(f_{S} \wedge f_{S'}),(t_{S} \times t_{S'}))$,
  where the transition function computes a pair of independent transitions,
  \[(t_{S} \times t_{S'})(s,s',a,a') = (t_{S}(s,a), t_{S'}(s',a')).\]
  The sequential synchronization of two deterministic finite automata,
  \[(S,A,B,i_{S},f_{S},t_{S})\mbox{ and }(T,B,C,i_{T},f_{T},t_{T}),\]
  is the automaton $(S \times T,A,C,(i_{S},i_{T}),(f_{S} \wedge f_{T}),(t_{S} \wedge t_{T}))$,
  where the transition function $(t_{S} \wedge t_{T})$ uses the output of the first transition as the input of the second
  \[(t_{S} \wedge t_{T})(s,t,a) = (s',t',c) \mbox{ where }
    t_{S}(s,a) = (s',b) \mbox{ and } t_{T}(t,b) = (t',c).\]
  We consider Mealy deterministic automata up to isomorphism of their state space.
  Mealy deterministic finite automata form a symmetric monoidal category, $\MealyAut$, with sequential composition and product as defined above.
\end{definition}

This construction, together with the results of~\Cref{sec:structured-st-construction}, leads to the following result.

\begin{proposition}\label{prop:mealydeterministicautomata}
  The category of Mealy deterministic finite automata is the free structured feedback category over the isomorphisms of automaton state spaces $\mathsf{U}_{\mathsf{Aut}} \colon \mathbf{Iso}(\AutSt) \to \mathbf{FinSet}$, that is,
  \[\MealyAut \cong \GSt(\mathbf{FinSet},\mathsf{U}_{\mathsf{Aut}}).\]
\end{proposition}

\subsection{Automata in Span(Graph)}

Our final construction is the canonical category of automata over the algebra of predicates given by spans.
These automata are analogous to the previous transition systems in $\SpanGraphV$, but their state space contains an initial state and a set of final states.
A similar definition has appeared previously in the literature for the modeling of Petri nets~\cite{Rathke2014PN}.

\begin{definition}\label{def:spanautomata}
  \defining{linkspanautomaton}{}
  A span automaton with left labels \(A\) and right labels \(B\)
  \[\mathbf{X} = (E_{X},S_{X},A,B,s_{X},t_{X},l_{X},r_{X},i_{X},f_{X})\]
  is given by a finite set of states $S_{X}$, a finite set of transitions $E_{X}$ with source and target $s_{X}, t_{X} \colon E_{X} \to S_{X}$ and with left \(l_{X} \colon E_{X} \to A\) and right \(r_{X} \colon E_{X} \to B\) labels, an initial state $i_{X} \in S_{X}$, and a subset of final states $f_{X} \colon S_{X} \to 2$.
  The product of two span automata
  \begin{align*}
    \mathbf{X} &= (E_{X},S_{X},A,B,s_{X},t_{X},l_{X},r_{X},i_{X},f_{X})\mbox{ and }\\
    \mathbf{Y} &= (E_{Y},S_{Y},A',B',s_{Y},t_{Y},l_{X},r_{X},i_{Y},f_{Y}),
  \end{align*}
  is given by the component-wise product
  \begin{multline*}
    \mathbf{X} \tensor \mathbf{Y} \coloneqq (E_{X} \times E_{Y}, S_{X} \times S_{Y}, A \times A', B \times B',\\
    s_{X} \times s_{Y}, t_{X} \times t_{Y}, l_{X} \times l_{Y}, r_{X} \times r_{Y}, i_{X} \times i_{Y}, f_{X} \wedge f_{Y}).
  \end{multline*}
  The composition of two span automata
  \begin{align*}
    \mathbf{X} &= (E_{X},S_{X},A,B,s_{X},t_{X},l_{X},r_{X},i_{X},f_{X})\mbox{ and }\\
    \mathbf{Y} &= (E_{Y},S_{Y},B,C,s_{Y},t_{Y},l_{X},r_{X},i_{Y},f_{Y}),
  \end{align*}
  is given by a pullback
  \[\mathbf{X} \comp \mathbf{Y} \coloneqq (E_{X} \times_{B} E_{Y}, S_{X} \times S_{Y},A,C,s_{X} \times s_{Y}, t_{X} \times t_{Y},l_{X},r_{Y},i_{X} \times i_{Y}, f_{X} \wedge f_{Y}),\]
  where \(E_{X} \times_{B} E_{Y}\) is the pullback of \(E_{X} \overset{r_{X}}{\to} B \overset{l_{Y}}{\leftarrow} E_{Y}\).
  We consider span automata up to isomorphism of their state space.
  Span automata form a symmetric monoidal category \(\SpanAut\) with sequential composition and the monoidal product defined above.
\end{definition}

This construction, together with the results of~\Cref{sec:structured-st-construction}, leads to the following result.

\begin{proposition}
The category of span automata is the free structured feedback category over the category of spans with the inclusion functor from automaton state spaces, $\mathsf{U}_{\mathsf{AutSpan}} \colon \mathbf{Iso}(\AutSt) \to \Span(\mathbf{FinSet})$, that is,
  \[ \SpanAut \cong \GSt(\Span(\mathbf{FinSet}),\mathsf{U}_{\mathsf{AutSpan}}).\]
\end{proposition}

\begin{example}
  By the universal property of the $\St(\bullet)$ construction, each Mealy automaton in $\MealyAut$ functorially induces a span automaton in $\SpanAut$ whose graph is the graph of the Mealy automaton.

  Consider the following Mealy automaton $\mathbf{X}$ in \Cref{figure:automatonspanmealy}, left. It has a state space $S_{X} = \{0,1,2\}$, left labels $A = \{a,b\}$ and trivial right labels $B = \mathbf{1}$, with initial state $i_{X} = 0$ and a unique final state $f_{X}(2) = \mathbf{true}$, while $f_{X}(0) = f_{X}(1) = \mathbf{false}$. Its transition function is given by $t_{X}(0,a) = 1$, $t_{X}(0,b) = 2$, $t_{X}(1,a) = 2$, $t_{X}(1,b) = 1$, $t_{X}(2,a) = 1$ and $t_{X}(2,b) = 1$.

  The corresponding span automaton $\mathbf{sX}$ (\Cref{figure:automatonspanmealy}, right) is not only the transition system, but also the markings for initial and final state. Explicitly, its set of edges -- or transitions -- is given by tuples
  \[E_{sX} = \{(0,a,1),(0,b,0),(1,a,1),(1,b,2),(2,a,1),(2,b,0)\},\]
  its state space is again $S_{sX} = \{0,1,2\}$, its left and right boundaries are again $A = \{a,b\}$ and $B = \mathbf{1}$, with initial state $i_{X} = 0$ and a unique final state $f_{X}(2) = \mathbf{true}$, while $f_{X}(0) = f_{X}(1) = \mathbf{false}$.

  \begin{figure}[h]
    \centering
\begin{tikzpicture}[x=0.75pt,y=0.75pt,yscale=-1,xscale=1]
\draw    (175,40) .. controls (177.69,26.22) and (207.61,25.54) .. (221.03,25.07) ;
\draw [shift={(223,25)}, rotate = 177.65] [color={rgb, 255:red, 0; green, 0; blue, 0 }  ][line width=0.75]    (10.93,-3.29) .. controls (6.95,-1.4) and (3.31,-0.3) .. (0,0) .. controls (3.31,0.3) and (6.95,1.4) .. (10.93,3.29)   ;
\draw   (163,50) .. controls (163,44.48) and (167.48,40) .. (173,40) .. controls (178.52,40) and (183,44.48) .. (183,50) .. controls (183,55.52) and (178.52,60) .. (173,60) .. controls (167.48,60) and (163,55.52) .. (163,50) -- cycle ;
\draw   (223,30) .. controls (223,24.48) and (227.48,20) .. (233,20) .. controls (238.52,20) and (243,24.48) .. (243,30) .. controls (243,35.52) and (238.52,40) .. (233,40) .. controls (227.48,40) and (223,35.52) .. (223,30) -- cycle ;
\draw   (203,90) .. controls (203,84.48) and (207.48,80) .. (213,80) .. controls (218.52,80) and (223,84.48) .. (223,90) .. controls (223,95.52) and (218.52,100) .. (213,100) .. controls (207.48,100) and (203,95.52) .. (203,90) -- cycle ;
\draw    (223,84) .. controls (237.56,81.28) and (242.14,56.29) .. (238.5,41.76) ;
\draw [shift={(238,40)}, rotate = 72.38] [color={rgb, 255:red, 0; green, 0; blue, 0 }  ][line width=0.75]    (10.93,-3.29) .. controls (6.95,-1.4) and (3.31,-0.3) .. (0,0) .. controls (3.31,0.3) and (6.95,1.4) .. (10.93,3.29)   ;
\draw    (223,35) .. controls (209.84,44.59) and (209.83,54.33) .. (212.77,76.29) ;
\draw [shift={(213,78)}, rotate = 262.22] [color={rgb, 255:red, 0; green, 0; blue, 0 }  ][line width=0.75]    (10.93,-3.29) .. controls (6.95,-1.4) and (3.31,-0.3) .. (0,0) .. controls (3.31,0.3) and (6.95,1.4) .. (10.93,3.29)   ;
\draw    (163,45) .. controls (135.58,26.61) and (128.7,73.88) .. (161.48,55.87) ;
\draw [shift={(163,55)}, rotate = 149.49] [color={rgb, 255:red, 0; green, 0; blue, 0 }  ][line width=0.75]    (10.93,-3.29) .. controls (6.95,-1.4) and (3.31,-0.3) .. (0,0) .. controls (3.31,0.3) and (6.95,1.4) .. (10.93,3.29)   ;
\draw    (243,35) .. controls (271.73,40.09) and (271.51,5.88) .. (244.26,24.13) ;
\draw [shift={(243,25)}, rotate = 324.72] [color={rgb, 255:red, 0; green, 0; blue, 0 }  ][line width=0.75]    (10.93,-3.29) .. controls (6.95,-1.4) and (3.31,-0.3) .. (0,0) .. controls (3.31,0.3) and (6.95,1.4) .. (10.93,3.29)   ;
\draw   (201,90) .. controls (201,83.37) and (206.37,78) .. (213,78) .. controls (219.63,78) and (225,83.37) .. (225,90) .. controls (225,96.63) and (219.63,102) .. (213,102) .. controls (206.37,102) and (201,96.63) .. (201,90) -- cycle ;
\draw    (201,90) .. controls (182.28,89.68) and (177.84,76.3) .. (173.54,62.7) ;
\draw [shift={(173,61)}, rotate = 72.38] [color={rgb, 255:red, 0; green, 0; blue, 0 }  ][line width=0.75]    (10.93,-3.29) .. controls (6.95,-1.4) and (3.31,-0.3) .. (0,0) .. controls (3.31,0.3) and (6.95,1.4) .. (10.93,3.29)   ;
\draw   (380,30) .. controls (380,18.95) and (388.95,10) .. (400,10) -- (470,10) .. controls (481.05,10) and (490,18.95) .. (490,30) -- (490,90) .. controls (490,101.05) and (481.05,110) .. (470,110) -- (400,110) .. controls (388.95,110) and (380,101.05) .. (380,90) -- cycle ;
\draw  [fill={rgb, 255:red, 0; green, 0; blue, 0 }  ,fill opacity=1 ] (402,50) .. controls (402,48.9) and (402.9,48) .. (404,48) .. controls (405.1,48) and (406,48.9) .. (406,50) .. controls (406,51.1) and (405.1,52) .. (404,52) .. controls (402.9,52) and (402,51.1) .. (402,50) -- cycle ;
\draw  [fill={rgb, 255:red, 0; green, 0; blue, 0 }  ,fill opacity=1 ] (462,28) .. controls (462,26.9) and (462.9,26) .. (464,26) .. controls (465.1,26) and (466,26.9) .. (466,28) .. controls (466,29.1) and (465.1,30) .. (464,30) .. controls (462.9,30) and (462,29.1) .. (462,28) -- cycle ;
\draw  [fill={rgb, 255:red, 0; green, 0; blue, 0 }  ,fill opacity=1 ] (452,96) .. controls (452,94.9) and (452.9,94) .. (454,94) .. controls (455.1,94) and (456,94.9) .. (456,96) .. controls (456,97.1) and (455.1,98) .. (454,98) .. controls (452.9,98) and (452,97.1) .. (452,96) -- cycle ;
\draw    (466,32) .. controls (480.71,36.85) and (479.01,8.77) .. (465.31,22.59) ;
\draw [shift={(464,24)}, rotate = 311.21] [color={rgb, 255:red, 0; green, 0; blue, 0 }  ][line width=0.75]    (4.37,-1.32) .. controls (2.78,-0.56) and (1.32,-0.12) .. (0,0) .. controls (1.32,0.12) and (2.78,0.56) .. (4.37,1.32)   ;
\draw    (459,96) .. controls (470.21,82.23) and (468.59,57.64) .. (464.33,33.83) ;
\draw [shift={(464,32)}, rotate = 79.54] [color={rgb, 255:red, 0; green, 0; blue, 0 }  ][line width=0.75]    (4.37,-1.32) .. controls (2.78,-0.56) and (1.32,-0.12) .. (0,0) .. controls (1.32,0.12) and (2.78,0.56) .. (4.37,1.32)   ;
\draw    (460,32) .. controls (447.08,52.6) and (443.43,67.61) .. (451.37,90.24) ;
\draw [shift={(452,92)}, rotate = 249.68] [color={rgb, 255:red, 0; green, 0; blue, 0 }  ][line width=0.75]    (4.37,-1.32) .. controls (2.78,-0.56) and (1.32,-0.12) .. (0,0) .. controls (1.32,0.12) and (2.78,0.56) .. (4.37,1.32)   ;
\draw    (406,46) .. controls (409.17,31.98) and (435.39,21.41) .. (458.25,25.65) ;
\draw [shift={(460,26)}, rotate = 192.43] [color={rgb, 255:red, 0; green, 0; blue, 0 }  ][line width=0.75]    (4.37,-1.32) .. controls (2.78,-0.56) and (1.32,-0.12) .. (0,0) .. controls (1.32,0.12) and (2.78,0.56) .. (4.37,1.32)   ;
\draw    (402,54) .. controls (387.16,68.84) and (385.95,36.06) .. (400.37,44.87) ;
\draw [shift={(402,46)}, rotate = 217.5] [color={rgb, 255:red, 0; green, 0; blue, 0 }  ][line width=0.75]    (4.37,-1.32) .. controls (2.78,-0.56) and (1.32,-0.12) .. (0,0) .. controls (1.32,0.12) and (2.78,0.56) .. (4.37,1.32)   ;
\draw    (449,96) .. controls (435.23,95.88) and (409.81,79.24) .. (404.38,55.82) ;
\draw [shift={(404,54)}, rotate = 79.54] [color={rgb, 255:red, 0; green, 0; blue, 0 }  ][line width=0.75]    (4.37,-1.32) .. controls (2.78,-0.56) and (1.32,-0.12) .. (0,0) .. controls (1.32,0.12) and (2.78,0.56) .. (4.37,1.32)   ;
\draw    (151,20) -- (165.75,38.44) ;
\draw [shift={(167,40)}, rotate = 231.34] [color={rgb, 255:red, 0; green, 0; blue, 0 }  ][line width=0.75]    (10.93,-3.29) .. controls (6.95,-1.4) and (3.31,-0.3) .. (0,0) .. controls (3.31,0.3) and (6.95,1.4) .. (10.93,3.29)   ;
\draw    (153,18) -- (149,22) ;
\draw    (400,24) .. controls (400.35,30.43) and (403.65,35.44) .. (403.98,44.13) ;
\draw [shift={(404,46)}, rotate = 270.73] [color={rgb, 255:red, 0; green, 0; blue, 0 }  ][line width=0.75]    (4.37,-1.32) .. controls (2.78,-0.56) and (1.32,-0.12) .. (0,0) .. controls (1.32,0.12) and (2.78,0.56) .. (4.37,1.32)   ;
\draw    (402,24) -- (398,24) ;
\draw    (380,55) -- (360,55) ;
\draw   (449,96) .. controls (449,93.24) and (451.24,91) .. (454,91) .. controls (456.76,91) and (459,93.24) .. (459,96) .. controls (459,98.76) and (456.76,101) .. (454,101) .. controls (451.24,101) and (449,98.76) .. (449,96) -- cycle ;
\draw (173,50) node  [font=\footnotesize]  {$0$};
\draw (233,30) node  [font=\footnotesize]  {$1$};
\draw (213,90) node  [font=\footnotesize]  {$2$};
\draw (179,12.4) node [anchor=north west][inner sep=0.75pt]  [font=\small]  {$a$};
\draw (166,80.4) node [anchor=north west][inner sep=0.75pt]  [font=\small]  {$b$};
\draw (269,17.4) node [anchor=north west][inner sep=0.75pt]  [font=\small]  {$a$};
\draw (244,60.4) node [anchor=north west][inner sep=0.75pt]  [font=\small]  {$a$};
\draw (196,50.4) node [anchor=north west][inner sep=0.75pt]  [font=\small]  {$b$};
\draw (126,40.4) node [anchor=north west][inner sep=0.75pt]  [font=\small]  {$b$};
\draw (477,30.4) node [anchor=north west][inner sep=0.75pt]  [font=\small]  {$a$};
\draw (471,61.4) node [anchor=north west][inner sep=0.75pt]  [font=\small]  {$a$};
\draw (411,13.4) node [anchor=north west][inner sep=0.75pt]  [font=\small]  {$a$};
\draw (388,27.4) node [anchor=north west][inner sep=0.75pt]  [font=\small]  {$b$};
\draw (405,82.4) node [anchor=north west][inner sep=0.75pt]  [font=\small]  {$b$};
\draw (438,47.4) node [anchor=north west][inner sep=0.75pt]  [font=\small]  {$b$};

\end{tikzpicture}
\caption{Mealy automaton and associated span automaton.}\label{figure:automatonspanmealy}
  \end{figure}
\end{example}

\section{Conclusions and Further Work}\label{sec:conclusions}\label{sec:newconclusions}
\subsection{Discussion}

We began this manuscript by pointing out the fragmented landscape of models for concurrent software.
We have now formally proven that any theory of predicates with a notion of feedback that accepts reasonable axioms (A1-A5) must already contain a simulation of $\SpanGraphV$: the algebra of open transition systems of Katis, Sabadini and Walters. This can be considered as being surprising given the dearth of results that establish the canonicity of models of concurrency. In this section we frame our contribution by expanding on its relevance for the wider research landscape.

\subsubsection{Why \emph{open} transition systems?}

Transition systems are the mainstay of concurrency theory: the firing semantics of a Petri net or the unfolding of the behaviour of a process algebra term are both important examples.
In recent years significant effort~\cite{Rathke2014PN,fong2019invitation,bonchi19} has been devoted to developing the mathematical foundations of \emph{compositional} modelling. Roughly speaking, this means the study of semantic foundations that include the ability to compose/decompose systems -- understood in a general sense -- into more primitive components in a homomorphic/functorial way, so that the behaviour of the composite is determined by that of the components. Conceptually, this is a strengthening of the foundational computer science principle of \emph{modularity}: our models do not merely need modular descriptions, these descriptions have to correspond homomorphically to a modular description of their behaviours. The $\SpanGraphV$ algebra is an application of the principle of compositionality for transition systems.

\smallskip
We live in an age of complexity, with enormous implications for modern software development. Software is nowadays rarely about sandbox development of a software product intended to be run on a single machine. Rather, software engineering is evolving to a discipline that deals complex networks consisting of interactions of libraries, services, APIs both locally and in the cloud. We believe that compositional approaches, given that they are a tool to tame this complexity, will become increasingly important in the future.

\subsubsection{Conceptual implications}

From an applied point of view, our results inform a minimal software architecture for a library that constructs and analyzes these modular transition systems, that are canonical models of concurrency. For instance, an object-oriented programmer may
\begin{itemize}[label=$\bullet$]
    \item define a specific class for ``theory of resources'', providing methods for the fundamental operations of composing, tensoring, identities and swapping, and providing subclasses for ``resources'' and ``processes'', if necessary;
    \item implement an abstract method that computes the $\St(\bullet)$ construction: this method will take a theory of processes and instantiate the free feedback category as a new theory of processes;
    \item work uniformly with multiple theories of processes and the theories of automata they generate, providing auxiliary methods (for reachability, connectedness, joining, \dots) that work across different theories of transition systems;
    \item thus, thanks to our results, obtain at the end of the day a graph -- in $\SpanGraph$ -- that can be understood as a transition system and analyzed with the same library.
\end{itemize}
A functional programmer may want to follow the architecture of the public Haskell implementation of the ideas of this paper \cite{spangraphroman}, where our running example (\Cref{fig:norintro}) is showcased.

The problem we solve is one of abstraction. When faced with the task of implementing models for transition systems, we could be tempted to simply implement each one of them separately.
Our result shows that a much more succinct choice is possible.
Reducing the lines of code is important: transition systems are usually formally analyzed to prove correctness; the less lines there are to analyze in the core implementation, the easiest it will be to be sure of its correctness.

Open transition systems allow us to construct transition systems from a few building blocks; compositionality eases both implementation and verification. Specifically, this could open an avenue to study the problem of compositional verification, where the automatic analysis of a global system must be modularly reduced to the analysis of its components.

\subsection{Conclusion}
We have characterized $\SpanGraphV$, an algebra of open transition systems, as the free \categoryWithFeedback{} over the category of spans of functions.
To do so, we have used the  $\St(\bullet)$ construction, characterized as the free \categoryWithFeedback{} in \cite{katis02}.
We have given this characterization more generally, for any category \(\catC\) with finite limits:
the category \({\Span(\Graph(\catC))}_{\ast}\) of spans of graphs in \(\catC\) is the free \categoryWithFeedback{} over the category of spans in \(\catC\).
Finally, we have defined a generalization of feedback categories to capture automata with initial and final states.

Further work will look at timed~\cite{cherubini2004timing} and probabilistic~\cite{dFSW2011compositionalmarkov,dFSW2011compositionalmarkov2} versions of the Cospan/Span model to connect it with recent work on modeling probabilistic programs with feedback categories~\cite{monoidalStreams22}.
We also plan to investigate the relationship between generalized \hyperlink{linkcategorywithfeedback}{feedback categories} (\Cref{sec:generalizingfeedback}) to approaches based on guarded recursion~\cite{goncharov18} and coalgebras~\cite{clouston15,milius17}.

\bibliography{bibliography}

\end{document}